\documentclass[12pt]{amsart}
\usepackage{amsfonts,graphics,amsmath,amsthm,amsfonts,amscd,amssymb,amsmath, enumerate,latexsym,multicol,mathrsfs}
\usepackage{graphicx, hhline}
\usepackage[all]{xy}
\usepackage[usenames]{color}
\usepackage[colorlinks,
linkcolor=blue,
anchorcolor=black,
citecolor=green
]{hyperref}
\usepackage{fancyhdr}
\usepackage{epsfig,url}
\usepackage[top=30truemm,bottom=30truemm,left=30truemm,right=30truemm]{geometry}

\hypersetup{colorlinks=true}

\title{An abundance theorem for generalised pairs}
\author{Zhengyu Hu}
\date{2021/03/22}

\address{Mathematical Sciences Research Center, 
	Chongqing University of Technology, No.69 Hongguang Avenue, Chongqing, 400054, China}
\email{huzhengyu@cqut.edu.cn}

\newcommand{\Supp}[0]{{\operatorname{Supp}}}

\DeclareMathOperator{\Bir}{Bir}

\DeclareMathOperator{\mult}{mult}

\DeclareMathOperator{\Spec}{Spec}
\DeclareMathOperator{\ex}{Ex}

\newtheorem{thm}{Theorem}[section]

\newtheorem{lem}[thm]{Lemma}
\newtheorem{cor}[thm]{Corollary}
\newtheorem{prop}[thm]{Proposition}

\theoremstyle{definition}
\newtheorem{defn}[thm]{Definition}
\newtheorem{rem}[thm]{Remark}

\newtheorem{exa}[thm]{Example}
\newtheorem*{ack}{Acknowledgments}

\newtheorem*{claim*}{Claim}

\newcommand{\K}{\mathbb K}
\newcommand{\PP}{\mathbb P}

\newcommand{\Q}{\mathbb Q}
\newcommand{\R}{\mathbb R}
\newcommand{\Z}{\mathbb Z}
\newcommand{\bir}{\dashrightarrow}
\newcommand{\rddown}[1]{\left\lfloor{#1}\right\rfloor} 
\begin{document}

\maketitle

\begin{abstract}
We prove the finiteness of $B$-representations of generalised log canonical pairs. As a consequence, we prove that, the (relative) abundance for a generalised semi-log canonical pair is implied by the abundance for its normalisation. Furthermore, we obtain an abundance theorem for generalised dlt pairs with abundant data, which is a natural inductive step of generalised minimal model program.
\end{abstract}

\tableofcontents

\section{Introduction}\label{sec1}

Throughout this paper we work over a fixed algebraically closed field of characterstic zero.\\

\emph{Generalised pairs} (\emph{g-pairs} for short), firstly introduced in \cite{birkarzhang} to prove the effectivity of the Iitaka
fibration, are regarded as powerful tools and in recent years led us to a significant development in birational geometry. The prototype of generalised pairs is derived from a canonical bundle formula for a certain class of algebraic fibre spaces, namely lc-trivial fibrations. For a quick and recent survey of these objects we refer to \cite[3.1]{hu2}. The approach of \cite{birkarzhang} allows us to apply most philosophy of minimal model program to log pairs whose boundaries birationally contain nef parts. We like to remark that, in the category of \emph{Kawamata log terminal} (\emph{klt} for short) (g-)pairs, the nef part would not cause too much trouble as long as the boundaries possess some strict positivity, such as bigness.

It is highly nontrivial to extend known results for klt pairs to analogues for \emph{log canonical} (\emph{lc} for short) pairs, such as \cite{birkar-flip}\cite{haconxu}\cite{haconxu-lcc}\cite{hu}\cite{hashizumehu}\cite{has-nonvanishing}. To attack some important conjectures regarding log canonical singularities, we turn to consider g-pairs with data possessing some nice and inductive positivity. As we mentioned previously, a theory of \emph{generalised lc} (\emph{g-lc} for short) pairs would be a powerful tool for investigating lc singularities. \\

\noindent \textbf{Finiteness of $B$-representations.}
We start our study with \emph{$B$-birational maps} and \emph{$B$-representations}. These notions were firstly introduced in \cite{fujino1} and thoroughly studied in \cite{fujino-gongyo}\cite{haconxu}, which plays a crucial role to arrange inductive proofs in higher dimensional algebraic geometry. To carry out further research into lc pairs, we generalise these notions to g-sub-pairs.

\begin{defn}[$B$-birational maps]\label{defnmain-B-bir}
	Two g-sub-pairs $(X/S,B+M)$ and $(X'/S,B'+M')$ are \emph{$B$-birational over $S$} if there is a common log resolution over $S$
	$$
	\xymatrix{	
		&  (\overline{X},\overline{B}+\overline{M})  \ar[dl]_{\alpha}\ar[dr]^{\beta}&\\
		(X, B+M) \ar[dr]_{} \ar@{-->}[rr]^{\varphi}&   &  (X', B'+M') \ar[dl]_{} &\\
		&   S &  &\\
	} 
	$$
	and write 
	$$K_{\overline{X}}+\overline{B}+ \overline{M} = \alpha^* (K_X+B+M);~ K_{\overline{X}}+\overline{B}'+ \overline{M}' = \beta^*(K_{X'}+B'+M'), $$
	where we set $K_X=\alpha_*K_{\overline{X}}$ and $K_{X'}=\beta_*K_{\overline{X}}$,
	then $\overline{B}=\overline{B}'$ and $\overline{M} = \overline{M}'$. In this case, the rational map $\varphi:=(\beta \circ \alpha^{-1}): X \dashrightarrow X'$ is a \emph{$B$-birational map} (cf.\cite[Definition 1.5]{fujino1}). We omit $S$ when it is a point. 
\end{defn}

\begin{defn}[$B$-representations]\label{defnmain-B-rep}
	Let $(X,B+M)$ be a proper g-sub-pair, and let $m$ be a positive
	integer such  $m(K_X + B+M)$ is Cartier. A $B$-birational map $\sigma: X \bir X \in 
	\Bir(X, B,M)$ defines a linear automorphism of $H^0(X, m(K_X + B+M))$. Thus
	we get the group homomorphism
	$$
	\rho_m : \Bir(X, B,M) \to \mathrm{Aut}_{k}(H^0(X, m(K_X + B+M))).
	$$
	The homomorphism $\rho_m$ is called a $B$-representation for $(X, B+M)$. We sometimes simply denote $\rho_m(\varphi)$ by $\varphi^*$
	for $\varphi \in \Bir(X, B,M)$ if there is no danger of confusion. .
\end{defn}

We generalise \cite[Theorem 3.15]{fujino-gongyo}\cite[Theorem 1.2]{haconxu} to g-sub-lc pairs, which asserts the \emph{finiteness of $B$-representations}. We remark that neither the positivity (of log canonical divisors or data) nor the effectiveness of boundaries is needed. Before we state the finiteness theorem, we introduce the following definition which is crucial to the arguments in Subsections \ref{subsec-vertical} and \ref{subsec-horizontal}.

\begin{defn}[General data, =Definition \ref{defn-g-data}]
	Let $(X/Z, B+M)$ be a g-sub-lc pair or a g-slc pair (Definition \ref{defnmain-g-slc-pair}) with data $\overline{M}$ and $\overline{X} \to X$. We say $\overline{M}$ is a \emph{general datum} if
	\begin{enumerate}
		\item $(\overline{X},\overline{B}+\overline{M})$ is a quasi-projective log smooth resolution;
		
		\item The support of $\overline{M}$ contains no strata;
		
		\item  $(\overline{X},\overline{B} - \epsilon \overline{B}^{=1} +\overline{M})$ is log smooth and sub-klt as a sub-pair for every number $\epsilon >0$.
	\end{enumerate}
\end{defn}

\begin{thm}[\text{=Theorem \ref{thm-good-lc}, cf.\cite[Theorem 3.15]{fujino-gongyo}}]\label{thmmain-good-lc}
	Let $(X, B+M)$ be a proper g-sub-lc pair with general data $\overline{M}$. Let $m \ge 2$ be an integer such that $m(K_X+B+M), m\overline{M}$ are Cartier.
	Then $\rho_m(\Bir(X, B,M))$ is a finite group. 
\end{thm}

The idea to prove the above theorem is similar to \cite[Theorem 3.15]{fujino-gongyo}. We reduce it to the finiteness of $\widetilde{B}$-pluricanonical representations \cite[Theorem 3.9]{fujino-gongyo}. We also remark that, \cite{haconxu-lcc} provides a Hodge theoretic approach, which might also be able to work for g-pairs.

It is natural to state the above theorem in the relative setting for more applications. For the definition of relative $B$-representations, see Definitions \ref{defn-rel-B-rep} and \ref{defn-rel-B-rep-2}.

\begin{thm}[=Theorem \ref{thm-finiteness}]\label{thmmain-finiteness}
	Let $(X/Z, B+M)$ be a g-sub-lc pair with general data $\overline{M}$. Let $m \ge 2$ be an integer such that $m(K_X + B+M), m\overline{M}$ are Cartier.
	Then $\rho_m(\Bir(X/Z, B,M))$ is a finite group.
\end{thm}

\noindent \textbf{Generalised semi-log canonical pairs.}
An important ingredient in an inductive argument is to glue hypersurfaces as well as sections of them. To this end, we have to generalise slc pairs to g-slc pairs (Definition \ref{defn-g-slc-pair}) and define (log) abundant data (Definition \ref{defn-g-pair-abundant}) as follows.

\begin{defn}[\text{G-slc pairs, cf.\cite[Definitions 2.1-2.3]{fujino-slc}\cite[Definition 2.5]{fujino-gongyo}}]\label{defnmain-g-slc-pair}
	A \emph{generalised semi-log canonical pair} (\emph{g-slc pair} for short) consists of 
	\begin{itemize}
		\item a reduced $S_2$ scheme $X$, proper over $Z$. We assume that it is pure dimensional and normal crossing in codimension one,
		
		\item an effective Weil divisor $B$ and a Weil divisor $M$, both containing no components of the conductor on $X$, such that $K_X+B+M$ is $\R$-Cartier, and
		
		\item a b-nef and b-$\R$-Cartier divisor $\mathbf{M}$ descends to some projective birational log smooth model $\overline{X} \overset{\pi}{\to} X^\nu \overset{\nu}{\to} X$ where $ X^\nu=\coprod_i X_i \overset{\nu}{\to} X$ is the normalisation, and $\pi_* \overline{M}_i= M_i$ where $\overline{M}_i=\mathbf{M}_{\overline{X}}|_{\overline{X}_i}$ and $K_{X_i}+B_i+M_i=\nu^*(K_X+B+M)|_{X_i}$, 
	\end{itemize}
	which satisfies:
	\begin{enumerate}
		\item $(X_i,B_i+M_i)$ is g-lc with data $\overline{M}_i$ for every $i$; 
		
		
		\item  For any pair of strata $\overline{V}_i,\overline{V}_j \subset \overline{X}$ with their image $V$ a component of the conductor of $X$, and a common resolution such that the diagram commutes,
		$$
		\xymatrix{	
			&  W  \ar[dl]_{\alpha}\ar[dr]^{\beta}&\\
			\overline{V}_i \ar[dr]_{\nu \circ \pi} &   &  \overline{V}_j \ar[dl]^{\nu \circ \pi} &\\
			&   V &  &\\
		} 
		$$
		we have $\alpha^*(\overline{M}'|_{\overline{V}_i})=\beta^*(\overline{M}'|_{\overline{V}_j})$.
	\end{enumerate}
\end{defn}

Although Definition \ref{defnmain-g-slc-pair} seems a bit technical, its prototype appears naturally as the following example. Note that the proof is not trivial (See Lemma \ref{lem-slc-contraction}).

\begin{exa}[\text{cf.\cite[Example 2.6]{fujino-gongyo}}]
	Let $(X, B+M)$ be a g-lc pair where $X$ is quasi-projective $\Q$-factorial klt. We put $S = \rddown{B}$. Then $(S, B_S+M_S)$ is g-slc where $K_{S^\nu} + B_{S^\nu}+M_{S^\nu} = (K_X + B+M)|_{S^\nu}$ given by the generalised adjunction.  
\end{exa}

\begin{defn}[G-slc pairs with abundant data]\label{defnmain-g-pair-abundant}
	A g-slc pair $(X/Z,B+M)$ with data $\overline{M}$ is said to be \emph{with (log) abundant data} if $\overline{M}$ is abundant (resp. log abundant with respect to $(\overline{X},\overline{B}+\overline{M})$). 
\end{defn}

\noindent \textbf{Abundance for generalised pairs.}
Next we generalise \cite[Theorem 1.4]{fujino-gongyo}\cite[Theorem 1.4]{haconxu} to g-slc pairs. We follow the strategy of \cite{fujino1} to glue \emph{admissible sections} (Definition \ref{defn-adim-sec}). We also remark that, \cite{haconxu} follows a different strategy based on Koll\'ar's glueing theory to obtain \cite[Theorem 1.4]{haconxu}, which might also be able to work for g-pairs. As a result, we establish the abundance theorem for g-slc pairs.

\begin{thm}[\text{=Theorem \ref{thm-semiample-slc}}]\label{thmmain-semiample-slc}
	Let $(X/Z,B+M)$ be a g-slc NQC pair and $\nu : X^\nu \to X$ be the normalisation. Assume
	$K_{X^\nu} +B^\nu +M^\nu:= \nu^*(K_X +B+M)$ is semi-ample$/Z$. Then $K_X +B+M$ is semi-ample$/Z$.
\end{thm}

Finally, we obtain the following result for g-pairs with abundant dada (cf.\cite[Theorem 3.3]{hu}\cite[Corollary 1.5]{haconxu}). This type of theorem is known to be a natural inductive step of log canonical minimal
model program (see \cite{birkar-flip}\cite{hu}). We expect the following theorem to have miscellaneous applications in the study of log canonical singularities in the future.

\begin{thm}[\text{=Theorem \ref{thm-lc-flips-3}}]\label{thmmain-lc-flips-3}
	Let $(X/Z, B+M)$ be a $\Q$-factorial g-dlt NQC pair with abundant data. Suppose that

    $\bullet$ $K_X + B+M$ is nef$/Z$,

    $\bullet$ $(K_X + B+M)|_{S_i}$ is semi-ample$/Z$ for each irreducible component $S_i$ of $\rddown{B}$,

    $\bullet$ $K_X + B - \epsilon P+M$ is semi-ample$/Z$ for some divisor $P \geq 0$ with $\mathrm{Supp} P = \rddown{B}$ and for any sufficiently small number $\epsilon > 0$.

    Then, $K_X + B+M$ is semi-ample$/Z$.	
\end{thm}

{\noindent \textbf{Contents of the paper.}}
In Section \ref{sec2}, we collect definitions, notations and basic results. 
In Section \ref{sec3}, we define $B$-birational maps and $B$-representations in the setting of g-pairs, and prove Theorems \ref{thmmain-good-lc} and \ref{thmmain-finiteness}.
In Section \ref{sec4-2}, we prove Theorem \ref{thmmain-semiample-slc}. In Section \ref{sec4-3}, we prove Theorem \ref{thmmain-lc-flips-3}.

\begin{ack}
	The author thanks Chen Jiang, Yong Hu, Kenta Hashizume and Haidong Liu for discussions. He thanks Professor Caucher Birkar for comments which improve the results in Section \ref{sec4}.
\end{ack}

\section{Preliminaries}\label{sec2}
In this section we collect definitions and some important results. Throughout this paper all varieties and schemes are over an algebraically closed field of characteristic zero, a variety means a disjoint union of irreducible varieties and a divisor refers to an $\R$-Weil divisor unless stated otherwise. 

\subsection{Notations and definitions}\label{sec2-defn}
We collect some notations and definitions. 

\noindent \textbf{Conventions.}
We denote by $\mathbb{K}$ the rational number field $\Q$ or the real number field $\R$. A birational model of a normal variety $X$, often denoted by $X'$, means a variety admits a proper and birational morphism to $X$, and a divisor $D$ over $X$ means a divisor on a birational model of $X$. 

\noindent \textbf{Contractions.}
In this paper a \emph{contraction} refers to a projective morphism $f\colon X\to Y$ of varieties 
such that $f_*\mathcal{O}_X=\mathcal{O}_Y$. In particular, $f$ has connected fibres. Moreover, 
if $X$ is normal, then $Y$ is also normal. A contraction $f$ is \emph{small} if $f$ does not contract any divisor. A birational map $\pi: X \dashrightarrow Y $ is a \emph{birational contraction} if the inverse of $\pi$ does not contract divisors. Note that $\pi$ is not necessarily a morphism unless stated otherwise. 

\noindent  \textbf{Divisors.}
Let $X$ be a normal variety, and let $M$ be a divisor on $X$. 
We denote the coefficient of a prime divisor $D$ in $M$ by $\mult_DM$. 
Writing $M=\sum m_iM_i$ where 
$M_i$ are the distinct irreducible components, the notation $M^{\ge a}$ means 
$\sum_{m_i\ge a} m_iM_i$, that is, we ignore the components with coefficient $<a$. One similarly defines $M^{= a},M^{<a}$, etc.. 

By a $\K$-rational function we mean a formal product of finitely many rational functions with $\K$-exponent, namely $\varphi:=\prod_{i=1}^k \varphi_i^{\alpha_i}$ with $\alpha_i \in \K$ for all $i$. We denote its $\K$-Cartier divisor by $(\varphi):=\sum_{i=1}^k \alpha_i (\varphi_i)$. Given two $\R$-Cartier divisors $D,D'$ on $X$, we say $D,D'$ are $\K$-linearly equivalent, and denote by $D \sim_\K D'$, if there exists a $\K$-rational function $\varphi$ such that $D=D'+(\varphi)$.

Given a proper morphism $f:X \to Z$, we say $D,D'$ are $\K$-linearly equivalent (resp. linearly equivalent, equivalent) over $Z$ and denote by $D \sim_\K D'/Z$ (resp. $D\sim D'/Z$, $D=D'/Z$) if there exists an $\R$-Cartier divisor $D_Z$ on $Z$ such that $D \sim_\K D' + f^* D_Z$ (resp. $D \sim  D' + f^* D_Z$, $D= D' + f^* D_Z$).

\begin{defn}
	Given a birational map $\varphi: X \dashrightarrow X'$ and an $\R$-Cartier divisor $D$, we define its pull-back $$\varphi^* D := \alpha_*\beta^*D,$$
	where $X \overset{\alpha}{\leftarrow} \overline{X} \overset{\beta}{\to} X'$ is a common resolution and $\varphi= \beta \circ \alpha^{-1}$. It is obvious that the definition is independent of the choice of $\overline{X}$.
\end{defn}

\noindent \textbf{Base loci.}
Let $D$ be a pseudo-effective $\R$-Cartier divisor on a normal variety $X$, proper over a quasi-projective variety $Z$. The $\R$-stable base
locus of $D$ is defined as
$$
\mathbf{B}(|D/Z|_\R):= \bigcap \{ \mathrm{Supp}L|L \in |D/Z|_\R \}.
$$ 
If $|D/Z|_\R = \emptyset$, then we put by convention $\mathbf{B}(|D/Z|_\R)= X$. When $D$ is $\Q$-Cartier, we similarly define the $\Q$-stable base locus $\mathbf{B}(|D/Z|_\Q)$. It is obvious that $\mathbf{B}(|D/Z|_\Q)=\mathbf{B}(|D/Z|_\R)$ for any $\Q$-Cartier divisor $D$ (\cite[Lemma 3.5.3]{bchm}), and $\mathbf{B}(|D/Z|_\Q)=\emptyset$  if and only if $D$ is semi-ample$/Z$. When $Z$ is a point or $Z=\Spec A$ is affine, we omit $Z$ and simply denote by $\mathbf{B}(|D|_\R)$.

\noindent \textbf{b-divisors.}
We recall some definitions regarding b-divisors. Let $X$ be a variety. A b-divisor $\mathbf{D}$ of $X$ is a family
$\{\mathbf{D}_{X'}\}_{X'}$ of $\R$-Weil divisors indexed by all birational models $X'$ of $X$,
such that $\mu_*(\mathbf{D}_{X''}) = \mathbf{D}_{X'}$ if $\mu: X'' \to X'$ is a birational contraction.

In most cases we focus on a class of b-divisors but not in full generality. An \emph{$\K$-b-Cartier b-divisor} $\mathbf{M}$ is defined by the choice of  
a projective birational morphism 
$\overline{X} \to X$ from a normal variety and an $\K$-Cartier divisor $\overline{M}$ on $\overline{X}$ in the way that $\mathbf{M}_{X'}=\mu^*\overline{M}$ for any birational model $\mu: X' \to \overline{X}$. In this case we say that $\overline{M}$ \emph{represents} $\mathbf{M}$ or $\mathbf{M}$ \emph{descends} to $\overline{X}$. 

Given an $\K$-b-Cartier b-divisors $\mathbf{M}$ on $X$ represented by $M_Y$ and a surjective proper morphism $f: Y \to X$, we define the pull-back of $\mathbf{M}$ as the $\K$-b-Cartier b-divisors $f^*\mathbf{M}$ represented by $\overline{f}^* \overline{M}$ where $\overline{f}: \overline{Y} \to \overline{X}$ is induced by $f$, $\overline{X} \to X$ and $\overline{Y} \to Y$. 

An $\R$-b-Cartier b-divisor represented by some $\overline{X}\to X$ and $\overline{M}$ is \emph{b-nef} if $\overline{M}$ is 
nef. Similarly we define a \emph{b-nef and abundant} b-divisor if $\overline{M}$ is 
nef and abundant.

We say a few words about the relative case. Let $f : X \to Y$ be a proper surjective morphism of varieties and $\mathbf{D}_1,\mathbf{D}_2$ be $\R$-b-Cartier b-divisors on $X$ represented by $\overline{D_1},\overline{D_2}$ on $\overline{X} \to X$. We say $\mathbf{D}$ is \emph{$\mathbb{K}$-linearly equivalent over $Y$} and denote by $\mathbf{D}_1 \sim_\mathbb{K} \mathbf{D}_2/Y$ if there exists birational models $\phi:X' \to \overline{X}$ and $ Y' \to Y$ such that the induced map $X' \dashrightarrow Y'$ is a morphism and $\phi^* \overline{D_1} \sim_{\mathbb{K}} \phi^* \overline{D_2}/Y'$. The reader may want to check this is well-defined since it is independent of the choice of birational models. In particular, we say $\mathbf{D}$ is \emph{$\mathbb{K}$-linearly trivial over $Y$} if $\mathbf{D} \sim_{\mathbb{K}} 0/Y$.

\noindent \textbf{Pairs.} 
A \emph{sub-pair} $(X/Z,B)$ consists of a normal variety $X$, a proper morphism $X \to Z$ and an $\R$-divisor 
$B$ such that $K_X+B$ is $\R$-Cartier. We say $B$ is a \emph{pre-boundary}.
If the coefficients of $B$ are at most $1$ we say $B$ is a 
\emph{sub-boundary}, and if in addition $B\ge 0$, 
we say $B$ is a \emph{boundary}. A sub-pair $(X/Z,B)$ is called a \emph{pair} if $B$ is a boundary.
When $Z$ is not relevant we usually drop it
and do not mention it: in this case one can just assume $X \to Z$ is the identity. 
When $Z$ is a point we also drop it but say the pair is projective.

Let $\phi\colon W\to X$ be a log resolution of a sub-pair $(X,B)$. Let $K_W+B_W$ be the 
pulback of $K_X+B$. The \emph{log discrepancy} of a prime divisor $D$ on $W$ with respect to $(X,B)$ 
is $1-\mult_DB_W$ and it is denoted by $a(D,X,B)$.
We say $(X,B)$ is a \emph{sub-lc pair} (resp. \emph{sub-klt pair}) 
if $a(D,X,B)$ is $\ge 0$ (resp. $>0$) for every $D$. When $(X,B)$ 
is a pair we remove the sub and say the pair is lc, etc. Note that if $(X,B)$ is an lc pair, then 
the coefficients of $B$ necessarily belong to $[0,1]$.  

Let $(X,B)$ be a sub-pair. A \emph{non-klt place} of $(X,B)$ is a prime divisor $D$ on 
birational models of $X$ such that $a(D,X,B)\le 0$. A \emph{non-klt center} is the image on 
$X$ of a non-klt place. When $(X,B)$ is lc, a non-klt center is also called an 
\emph{lc center}. For definitions and standard results on singularities of pairs
we refer to \cite{kollar-mori}.

\noindent \textbf{Generalised pairs.}
For the basic theory of generalised (polarised) pairs we refer to \cite[Section 4]{birkarzhang}.
Below we recall some of the main notions and discuss some basic properties.

A \emph{generalised sub-pair} (\emph{g-sub-pair} for short) consists of 
\begin{itemize}
	\item a normal variety $X$ equipped with a proper
	morphism $X\to Z$, 
	
	\item an $\R$-divisor $B$ on $X$, and 
	
	\item a b-$\R$-Cartier b-divisor $\mathbf{M}$ represented 
	by some projective birational morphism $\overline{X} \overset{\phi}\to X$ and $\R$-Cartier divisor
	$\overline{M}$ on $X$ such that $\overline{M}$ is nef$/Z$ and $K_{X}+B+M$ is $\R$-Cartier,
	where $M := \phi_*\overline{M}$.
\end{itemize}
A generalised sub-pair is a \emph{generalised pair} (\emph{g-pair} or \emph{pair} for short) if $B$ is effective. 

We usually refer to the sub-pair by saying $(X/Z,B+M)$ is a \emph{g-sub-pair with data} $\overline{M}$ (or $\mathbf{M}$), and we omit the data if it is not important. Since a b-$\R$-Cartier b-divisor is defined birationally, in practice we will often replace $\overline{X}$ with a log resolution (and hence omit it) and replace $\overline{M}$ with its pullback. We say $(\overline{X},\overline{B}+\overline{M})$, where $K_{\overline{X}}+\overline{B}+\overline{M}$ is the pull-back of $K_X+B+M$, is \emph{log smooth} if $(\overline{X},\overline{B}+\overline{M})$ is log smooth (as a sub-pair) and $\mathbf{M}$ descends to $\overline{X}$. In this case, we say $(\overline{X},\overline{B}+\overline{M}) \to X$, is a \emph{log resolution} of $(X,B+M)$.
When $Z$ is not relevant we usually drop it
and do not mention it: in this case one can just assume $X \to Z$ is the identity. 
When $Z$ is a point we also drop it but say the pair is proper. A g-sub-pair naturally defines a b-divisor $\mathbf{K}+\mathbf{B}+\mathbf{M}$ which descends to $X$ so that for any projective birational morphism $X' \overset{\phi'}\to X$ we have $\phi'^*(K_X+B+M)=\mathbf{K}_{X'}+\mathbf{B}_{X'}+\mathbf{M}_{X'}$. We call $\mathbf{B}$ the \emph{pre-boundary b-divisor} and $\mathbf{M}$ the \emph{moduli b-divisor} or the \emph{data}. Similarly, If the coefficients of $\mathbf{B}$ are at most $1$ we say $\mathbf{B}$ is a 
\emph{sub-boundary b-divisor}, and if in addition $B\ge 0$, 
we say $\mathbf{B}$ is a \emph{boundary b-divisor}.

Given a g-sub-pair $(X/Z,B+M)$ with data $\overline{M}$, if $\overline{M}$
is a convex combination of $\Q$-Cartier divisors which are
nef over $Z$, then $\overline{M}$ is \emph{NQC} and the g-pair $(X/Z, B + M)$ is an \emph{NQC g-pair}. Here, NQC stands for \emph{nef $\Q$-Cartier combinations}.  

\noindent \textbf{Generalised singularities.} 
Now we define generalised singularities of a g-pair.
Replacing $X$ we can assume $\phi$ is a log resolution of $(X,B)$. We can write 
$$
K_{\overline{X}}+\overline{B}+\overline{M}=\phi^*(K_{X}+B+M)
$$
for some uniquely determined $B$. For a prime divisor $D$ on $X$ the \emph{generalised log discrepancy} 
$a(D,X,B+M)$ is defined to be $1-\mult_D\overline{B}$. 

A \emph{generalised non-klt center} of a g-sub-pair $(X,B+M)$ is the image of a prime divisor 
$D$ over $X$ with $a(D,X,B+M)\le 0$, and 
the \emph{generalised non-klt locus} of the g-sub-pair is the union of all the generalised non-klt centers. When $(X,B+M)$ is g-lc (resp. g-sub-lc), a generalised non-klt center is also called a 
\emph{g-lc centre} (resp. \emph{g-sub-lc centre}).

We say $(X,B+M)$ is 
\emph{generalised lc} or \emph{g-lc} (resp. \emph{generalised klt} or \emph{g-klt})
if for each $D$ the generalised log discrepancy $a(D,X,B+M)$ is $\ge 0$ (resp. $>0$).
A g-lc pair $(X/Z,B+M)$ with data $\overline{M}$ is \emph{generalised dlt} or \emph{g-dlt} if there is an open subset $U \subset X$ containing the generic points of all g-lc centres as lc centres of $(U,B|_U)$ and $(U,B|_U+M|_U)$ is log smooth. If in addition each connected component of $\rddown{B}$ is irreducible, we say the pair is \emph{generalised plt} or \emph{g-plt} for short. Note that \emph{g-sub-lc, g-sub-klt,} etc. can be defined similarly when we drop the assumption of the effectivity of $B$.

The lemmas below show that the above definition of g-dlt pairs is a natural generalisation of dlt pairs. See also \cite[Remark 2.7]{has-nonvanishing}. This answers a question of J. Han and Z. Li \cite[Remark 2.4]{hanli} for quasi-projective varieties. 

\begin{lem}\label{dlt-resolution}
	Let $(X,B+M)$ be a $\Q$-factorial g-dlt pair with data $\overline{M}$. Then there exists a log resolution $\pi:(\overline{X},\overline{B}+\overline{M}) \to X$ with $a(E,X,B+M)>0$ for every exceptional$/X$ divisor $E$ on $\overline{X}$. In particular, $\pi$ is isomorphic near the generic point of every stratum.
\end{lem}
\begin{proof}
	We first assume $X$ is quasi-projective. Pick a log resolution $\pi: \overline{X} \to X$, a number $0< \epsilon \ll 1$ and write 
	$$
	K_{\overline{X}} +\overline{B}_\epsilon +\overline{M}=\pi^*(K_X+(1-\epsilon )B +M) +F
	$$
	where $\overline{B}_\epsilon,F \ge 0$ have no common components. Set 
	$$
	N=K_{\overline{X}} +\overline{B}_\epsilon -\overline{M}/X
	$$
	such that $N$ is exceptional$/X$. Pick an exceptional divisor $E\ge 0$ such that 
	\begin{itemize}
		\item $-E$ is ample$/X$;
		
		\item $N+E$ is $\Q$-Cartier;
		
		\item $\|E\| \ll 1$. 
	\end{itemize}
    Now we run an MMP$/X$ on $K_{\overline{X}} +\overline{B}_\epsilon+E +\alpha \overline{M} $. By proof of \cite[Theorem 1.4]{bhzariski}, the above MMP is $\overline{M}$-trivial. Moreover, by the negativity lemma \cite[Theorem 3.5]{birkar-flip} and \cite{bchm}, the MMP terminates and contracts all exceptional divisors with zero log discrepancies, hence $\mathbf{M}$ descends to the canonical model $\overline{X}'$ over $X$. Replacing $\overline{X}$ with a suitable resolution of $\overline{X}'$ we obtain the required log resolution.
    
    Now we treat the general case. Since the canonical model $\overline{X}'$ of $({\overline{X}}/X, \overline{B}_\epsilon+E +\alpha \overline{M})$ still exists, again replacing $\overline{X}$ with a suitable resolution of $\overline{X}'$ we conclude the lemma.
\end{proof}

The following proof is introduced to me by Chen Jiang.

\begin{lem}\label{lem-g-dlt}
	Let $(X,B+M)$ be a quasi-projective g-dlt pair with data $\overline{M}$. Then $X$ admits a small $\Q$-factorialisation which is isomorphic near the generic point of every stratum.  
\end{lem}
\begin{proof}
	Let $U$ be an open subset $U \subset X$ containing the generic points of all g-lc centres with $(U,B|_U+M|_U)$ log smooth. Then there exists a log smooth sub-pair $\pi:(X',B'+M') \to X$ such that $\pi$ is an isomorphism over $U$. Now run an MMP$/X$ on $K_{X'}+B'+E'$ with scaling of an ample$/X$ divisor, where $E' \ge -B'^{<0}$ contains all exceptional divisors in its support, and $(X',B'+E')$ is dlt. Since $K_{X'}+B'+E'=E'-M'/X$ and $M'$ is the push-down of a nef and big$/X$ divisor, after replacing we can take $M' \ge 0$ and $\mult_{E_i} M'$ sufficiently small for any irreducible component $E_i$ of $E'$. So, this MMP contracts $E'$ after finitely many steps by the negativity lemma \cite[Lemma 3.3]{birkar-flip}, hence reach a $\Q$-factorial model $X''$ which gives a small contraction over $X$. Note that $X'' \to X$ is an isomorphism over $U$. Therefore, by our definition, $(X'',B''+M'')$ is $\Q$-factorial g-dlt.
\end{proof}


\begin{lem}\label{lem-g-dlt-2}
	Let $(X,B+M)$ be a g-dlt pair. Then, any stratum is normal.
\end{lem}
\begin{proof}
	Since the proof of \cite[Proposition 2.43]{kollar-mori} works verbatim, we can assume $(X,B+M)$ be g-plt. Then the lemma follows from the connectedness lemma.
\end{proof}

\noindent \textbf{Generalised adjunctions.}
Let $(X, B + M)$ be a g-pair with data $\overline{M}$. Assume that $S$ is the normalisation of a component of $\rddown{B}$ and $\overline{S}$ is its birational transform on $\overline{X}$.  Let
$K_{\overline{S}} + B_{\overline{S}} + M_{\overline{S}}^{\psi} := (K_{\overline{X }}+ \overline{B} + \overline{M}+(\psi))|_{\overline{S}}$
where $B_{\overline{S}} = (\overline{B} -\overline{S})|_{\overline{S}}$ and $M_{\overline{S}}^{\psi} = (\overline{M}+(\psi))|_{\overline{S}}$ for some $\K$-rational function $\psi$. Let $B_S,M_S^{\psi}$ be the push forward of $B_{\overline{S}},M_{\overline{S}}^{\psi}$. Then we get the equality
$$
K_S + B_S + M_S^{\psi} = (K_X + B + M+(\psi))|_S
$$
which we refer to as generalised adjunction. Now assume that $(X, B + M)$ is g-lc (resp. g-dlt, g-plt). It is also clear that $(S, B_S + M_S^\psi)$ is also g-lc (resp. g-dlt, g-klt).

We take $(\psi)=0$ if $\overline{M}$ does not contain $\overline{S}$ in its support. In this case, we simply denote $M_S^{\psi}$ by $M_S$.

\noindent \textbf{Minimal models.}
A g-pair $(Y/Z,B_Y+M_Y)$ with data $M_{\overline{Y}}$ is a \emph{log birational model} of a g-pair $(X/Z,B+M)$ with data $\overline{M}$ if we are given a birational map
$\phi\colon X\bir Y$, $B_Y=B^\sim+E$ where $B^\sim$ is the birational transform of $B$ and
$E$ is the reduced exceptional divisor of $\phi^{-1}$, that is, $E=\sum E_j$ where $E_j$ are the
exceptional/$X$ prime divisors on $Y$ and $\overline{M}_Y=\overline{M}$. 

A log birational model $(\overline{X}/Z,\overline{B}+\overline{M})$ is a \emph{log smooth model} of $(X/Z,B+M)$ if it is log smooth with data $\overline{M}$.

A log birational model $(Y/Z,B_Y+M_Y)$ is a \emph{weak log canonical (weak lc for short) model} of $(X/Z,B+M)$ if

$\bullet$ the data $M_{\overline{Y}}=\overline{M}$,

$\bullet$ $K_Y+B_Y+M_Y$ is nef$/Z$, and

$\bullet$ for any prime divisor $D$ on $X$ which is exceptional/$Y$, we have
$$
a(D,X,B+M)\le a(D,Y,B_Y+M_Y).
$$

A weak lc model $(Y/Z,B_Y+M_Y)$ is a \emph{minimal model} of $(X/Z,B+M)$ if

$\bullet$ $Y$ is $\Q$-factorial,

$\bullet$ the above inequality on log discrepancies is strict.

A minimal model $(Y/Z, B_Y+M_Y)$ is a \emph{log minimal model} if

$\bullet$ $(Y/Z,B_Y+M_Y)$ is g-dlt.

A minimal model $(Y/Z, B_Y+M_Y)$ is a \emph{good minimal model} if $K_Y + B_Y +M_Y$ is semi-ample$/Z$. In this case, $K_Y + B_Y +M_Y$ defines a contraction $g:Y \to W$ such that $K_Y + B_Y +M_Y=g^*A_W$ for some ample$/Z$ divisor $A_W$. We say $W$ is the \emph{canonical model} of $(X/Z,B+M)$.  \\

On the other hand, a log birational model $(Y/Z,B_Y+M_Y)$  is called a \emph{weak Mori fibre space} of $(X/Z,B+M)$ if

$\bullet$ there is a $K_Y+B_Y+M_Y$-negative extremal contraction $Y\to T$
with $\dim Y>\dim T$, and

$\bullet$ for any prime divisor $D$ (on birational models of $X$) we have
$$
a(D,X,B+M)\le a(D,Y,B_Y+M_Y)
$$
and strict inequality holds if $D$ is
on $X$ and contracted$/Y$.

A weak Mori fibre space $(Y/Z,B_Y+M_Y)$ is a \emph{Mori fibre space} of $(X/Z,B+M)$ if

$\bullet$ $(Y/Z,B_Y+M_Y)$ is $\Q$-factorial g-dlt.

\subsection{Nef and abundant divisors}\label{sec2-dim}
In this subsection we will introduce the notion of nef and abundant divisor and elementary properties in the setting of $\R$-divisors. Most contents are taken from \cite{hu2}.

\noindent \textbf{Iitaka dimension and numerical dimension.}
Recall the following definitions of Iitaka dimension and numerical dimension. Both integers are birational invariants given by the growth of the quantity of sections.

\begin{defn}[Invariant Iitaka dimension]\label{defn--inv-iitaka-dim}
	Let $X$ be a normal complete variety, and $D$ be an $\mathbb{R}$-Cartier divisor $D$ on $X$. 
	We define the {\em invariant  Iitaka dimension} of $D$, denoted by $\kappa_{\iota}(X,D)$, as follows (see also \cite[Definition 2.5.5]{fujino-book}):  
	If there is an $\mathbb{R}$-divisor $E\geq 0$ such that $D\sim_{\mathbb{R}}E$, set $\kappa_{\iota}(X,D)=\kappa(X,E)$. 
	Here, the right hand side is the usual Iitaka dimension of $E$. 
	Otherwise, we set $\kappa_{\iota}(X,D)=-\infty$. 
	We can check that $\kappa_{\iota}(X,D)$ is well-defined, i.e., when there is $E\geq 0$ such that $D\sim_{\mathbb{R}}E$, the invariant Iitaka dimension $\kappa_{\iota}(X,D)$ does not depend on the choice of $E$. 
	By definition, we have $\kappa_{\iota}(X,D)\geq0$ if and only if $D$ is $\mathbb{R}$-linearly equivalent to an effective $\mathbb{R}$-divisor. 
	
	Let $X\to Z$ be a proper morphism from a normal variety to a variety, and let $D$ be an $\mathbb{R}$-Cartier divisor on $X$. 
	Then the {\em relative invariant Iitaka dimension} of $D$, denoted by $\kappa_{\iota}(X/Z,D)$, is defined by $\kappa_{\iota}(X/Z,D)=\kappa_{\iota}(X,D|_{F})$, where $F$ is a very general fibre of the Stein factorisation of $X\to Z$.
	Note that the value $\kappa_{\iota}(X,D|_{F})$ does not depend on the choice of $F$ (see \cite[Lemma 2.10]{hashizumehu}). 
\end{defn}

\begin{defn}[Numerical dimension]\label{defn--num-dim}
	Let $X$ be a normal projective variety, and $D$ be an $\mathbb{R}$-Cartier divisor $D$ on $X$. 
	We define the {\em numerical dimension} of $D$, denoted by $\kappa_{\sigma}(X,D)$, as follows (see also \cite[V, 2.5 Definition]{nakayama}): 
	For any Cartier divisor $A$ on $X$, we set
	$$
	\sigma(D;A)={\rm max}\left\{k\in \mathbb{Z}_{\geq0}\middle|\, \underset{m\to \infty}{\rm lim}{\rm sup}\frac{{\rm dim}H^{0}(X,\mathcal{O}_{X}(\llcorner mD \lrcorner+A))}{m^{k}}>0\right\}
	$$
	if ${\rm dim}H^{0}(X,\mathcal{O}_{X}(\llcorner mD \lrcorner+A))>0$ for infinitely many $m\in \mathbb{Z}_{>0}$, and otherwise we set $\sigma(D;A):=-\infty$. 
	Then, we define 
	$$\kappa_{\sigma}(X,D):={\rm max}\{\sigma(D;A)\,|\,A{\rm\; is\; a\;Cartier\;divisor\;on\;}X\}.$$
	
	Let $X\to Z$ be a projective morphism from a normal variety to a variety, and let $D$ be an $\mathbb{R}$-Cartier divisor on $X$. 
	Then, the {\em relative numerical dimension} of $D$ over $Z$, denoted by $\kappa_{\sigma}(X/Z,D)$, is defined by $\kappa_{\sigma}(F,D|_{F})$, where $F$ is a very general fibre of the Stein factorisation of $X\to Z$.  
	We note that the value $\kappa_{\sigma}(F,D|_{F})$ does not depend on the choice of $F$, so the relative numerical dimension is well-defined. 
\end{defn}

For a collection of basic properties of the invariant Iitaka dimension and the numerical dimension, we refer to \cite[Remark 2.8]{hashizumehu}. 

\begin{defn}[Relatively abundant divisor and relatively log abundant divisor]\label{defn--abund}
	Let $f\colon X\to Z$ be a projective morphism from a normal variety to a variety, and $D$ be an $\mathbb{R}$-Cartier divisor on $X$. 
	We say that $D$ is {\em abundant over} $Z$ if the equality $\kappa_{\iota}(X/Z,D)=\kappa_{\sigma}(X/Z,D)$ holds. 
	When $Z$ is a point, we simply say $D$ is {\em abundant}. 
	
	Let $f\colon X\to Z$ and $D$ be as above, and $(X,B+M)$ be a g-sub-lc sub-pair. 
	We say that $D$ is $\pi$-{\em log abundant} (or {\em log abundant over} $Z$) with respect to $(X,B+M)$ if $D$ is abundant over $Z$ and for any g-lc center $S$ of $(X,B+M)$ with the normalization $S^{\nu}\to S$, the pullback $D|_{S^{\nu}}$ is abundant over $Z$. 
\end{defn}

\noindent \textbf{Nef and abundant divisors.}
Recall that, given a a proper morphism $f\colon X\to Z$ from a normal variety to a variety, an $\R$-Cartier divisor $D$ is \emph{semi-ample} over $Z$ if there exist a proper surjective morphism $h: X \to Y$ over $Z$ and an ample divisor $D_Y$ of $Y$ such that $D \sim_\R h^*D_Y$.

\begin{lem}[\text{\cite[Lemma 2.7]{hu2}}]\label{lem-semi-ample-divisor}
	Notation as above, let $D$ be an $\R$-Cartier divisor. 
	\begin{enumerate}
		\item $D$ is semi-ample$/Z$ if and only if $D$ is a convex combination of semi-ample$/Z$ $\Q$-divisors.
		
		\item Let $D'$ be another $\R$-Cartier divisor. If $D,D'$ are semi-ample$/Z$, then so is $D+D'$.
		
		\item Let $g:W \to X$ be a proper surjective morphism. Then, $D$ is semi-ample$/Z$ if and only if $g^*D$ is semi-ample$/Z$.
	\end{enumerate}
	 
\end{lem}

\begin{lem}[Nef and abundant divisor, \text{\cite[Lemma 2.8]{hu2}}]\label{lem-nef-abundant-divisor}
	Let $f\colon X\to Z$ be a projective morphism from a normal variety to a variety, and let $D$ be a nef$/Z$ $\R$-Cartier divisor on $X$. Then, the following conditions are equivalent:
	\begin{enumerate}
		\item $D$ is abundant over $Z$.
		
		\item There exist a birational model $\pi: X' \to X$, a surjective morphism $g : X' \to Y$ of smooth quasi-projective varieties over $Z$,
		and a nef and big$/Z$ divisor $B$ of $Y$ such that $\pi^*D \sim_\R g^*B$.
		
		\item Given a sufficiently general fibre $F$, the asymptotic vanishing order $o_\Gamma(D|_F)=0$ for every prime divisor $\Gamma$ over $F$.
	\end{enumerate} 
\end{lem}

\begin{lem}\label{lem-nef-abund-1}
	Let $f:X \to Z$ be a proper morphism from a normal variety to a variety, $g:Y \to X$ be a proper surjective morphism of normal varieties, and $D$ be an $\R$-Cartier divisor on $X$. Then, $D$ is nef and abundant over $Z$ if and only if $f^*D$ is nef and abundant over $Z$.
\end{lem}

\begin{lem}\label{lem-nef-abund-2}
		Let $f\colon X\to Z$ be a projective morphism from a normal variety to a variety, and let $D$ be an $\mathbb{R}$-Cartier divisor on $X$. Suppose $D$ is a non-negative $\R$-linear combination of finitely many $\Q$-Cartier divisors which are nef and abundant over $Z$. Then, $D$ is nef and abundant over $Z$. 
\end{lem}

\begin{defn}[b-nef and log abundant divisors]\label{defn-b-nef-log-abundant-divisor}
	Let $(X/Z,B+M)$ be a g-sub-lc sub-pair with data $\mathbf{M}$, and $\mathbf{D}$ be an $\mathbb{R}$-b-Cartier b-divisor on $X$. We say $\mathbf{D}$ is b-nef and log abundant if for any log resolution $(\overline{X},\overline{B}+\overline{M}) \to X$ to which $\mathbf{D},\mathbf{M}$ descends, we have $\mathbf{D}_{\overline{X}}$ is nef and log abundant. Note that the definition is independent of the choice of $(\overline{X},\overline{B}+\overline{M})$.
\end{defn}

\subsection{Generalised semi-log canonical pairs}
We generalise the notion of \emph{slc pairs} and \emph{sdlt pairs} to the setting of generalised pairs.

\begin{defn}[\text{G-slc pairs, cf.\cite[Definitions 2.1-2.3]{fujino-slc}\cite[Definition 2.5]{fujino-gongyo}}]\label{defn-g-slc-pair}
	A \emph{generalised semi-log canonical pair} (\emph{g-slc pair} for short) consists of 
	\begin{itemize}
		\item a reduced $S_2$ scheme $X$, proper over $Z$. We assume that it is pure dimensional and normal crossing in codimension one,
		
		\item an effective Weil divisor $B$ and a Weil divisor $M$, both containing no components of the conductor on $X$, such that $K_X+B+M$ is $\R$-Cartier, and
		
		\item a b-nef and b-$\R$-Cartier divisor $\mathbf{M}$ descends to some projective birational log smooth model $\overline{X} \overset{\pi}{\to} X^\nu \overset{\nu}{\to} X$ where $ X^\nu=\coprod_i X_i \overset{\nu}{\to} X$ is the normalisation, and $\pi_* \overline{M}_i= M_i$ where $\overline{M}_i=\mathbf{M}_{\overline{X}}|_{\overline{X}_i}$ and $K_{X_i}+B_i+M_i=\nu^*(K_X+B+M)|_{X_i}$, 
	\end{itemize}
    which satisfies:
    \begin{enumerate}
    	\item $(X_i,B_i+M_i)$ is g-lc with data $\overline{M}_i$ for every $i$; 
    	
    	
    	\item  For any pair of strata $\overline{V}_i,\overline{V}_j \subset \overline{X}$ with their image $V$ a component of the conductor of $X$, and a common resolution such that the diagram commutes,
    	$$
    	\xymatrix{	
    		&  W  \ar[dl]_{\alpha}\ar[dr]^{\beta}&\\
    		\overline{V}_i \ar[dr]_{\nu \circ \pi} &   &  \overline{V}_j \ar[dl]^{\nu \circ \pi} &\\
    		&   V &  &\\
    	} 
    	$$
    	we have $\alpha^*(\overline{M}'|_{\overline{V}_i})=\beta^*(\overline{M}'|_{\overline{V}_j})$.
    \end{enumerate}
\end{defn}
    
    Similarly, we usually refer to the sub-pair by saying $(X/Z,B+M)$ is a g-slc-pair with data $\mathbf{M}$ or $\overline{M}$, and we omit the data if it is not important. When $Z$ is not relevant we usually drop it
    and do not mention it.
    When $Z$ is a point we also drop it but say the pair is proper. We say it is \emph{NQC} if its data $\overline{M}$ is a convex combination of $\Q$-Cartier nef$/Z$ divisors.
 
 Although Definition \ref{defn-g-slc-pair} seems a bit technical, its prototype appears naturally as the following example.
 
\begin{exa}[\text{cf.\cite[Example 2.6]{fujino-gongyo}}]
	Let $(X, B+M)$ be a g-lc pair where $X$ is quasi-projective $\Q$-factorial klt. We put $S = \rddown{B}$. Then $(S, B_S+M_S)$ is g-slc where $K_{S^\nu} + B_{S^\nu}+M_{S^\nu} = (K_X + B+M)|_{S^\nu}$ given by the generalised adjunction. See Lemma \ref{lem-slc-contraction} for a proof.
\end{exa}
    
\begin{defn}[\text{G-sdlt pairs}]
	Notation as above, a g-slc pair $(X,B+M)$ is a \emph{generalised semi-divisorial log terminal pair} (\emph{g-sdlt pair}, for short) if 
	\begin{enumerate}
		\item $X_i$ is normal, that is, $X_i^\nu$ is isomorphic to $X_i$;
		
		\item $(X_i, B_i+M_i)$ is g-dlt.
	\end{enumerate}
\end{defn}


\begin{lem}\label{lem-g-sdlt}
	Notation as in Definition \ref{defn-g-slc-pair}, we have:
	\begin{enumerate}
		\item There exists a Weil divisor $M' \sim_\R M$ such that the corresponding datum $\overline{M}'$ is log smooth and contains no strata. If $n\overline{M}$ is Cartier, then we can further require $nM \sim nM'$ with the coefficients of $\overline{M}' \le 1/n$.
		
		\item Suppose $(X,B+M)$ is quasi-projective g-sdlt. Then, there exist data $\mathbf{M}' \sim_\R \mathbf{M}$ and $\overline{X} \to X$ such that Condition (2) holds for every pair of strata with the same image on $X$.
		
		\item Suppose $(X,B+M)$ is quasi-projective g-sdlt and $S \le \rddown{B}$ is a reduced divisor. Then, $(S,B_S+M_S)$ is well-defined and g-sdlt. 
	\end{enumerate} 
\end{lem}
\begin{proof}
    Since we can assume $\overline{X}$ is quasi-projective, write $\overline{M}=H-L$ with $H,L$ supported by general hyperplane sections. Note that, by a theorem of Bertini, we may require:
	\begin{itemize}
		\item The support $H$ and $L$ contains no strata;
		
		\item $H|_{\overline{V}_i} = H|_{\overline{V}_j},L|_{\overline{V}_i} = L|_{\overline{V}_j}$ for any pair of strata $\overline{V}_i,\overline{V}_j \subset \overline{X}$ with their image $V$ a component of the conductor. 
	\end{itemize}
    Since $X$ is $S_2$, $H-L$ descend to a divisor on $X$. Hence we prove (1). Since $\pi$ can be chosen to be isomorphic at the generic point of every stratum by Lemmas \ref{dlt-resolution} and \ref{lem-g-dlt}, one immediately obtains (2). Finally, (3) follows from (2) and Lemma \ref{lem-g-dlt-2}.
\end{proof}

\begin{defn}[G-slc pairs with abundant data]\label{defn-g-pair-abundant}
	A g-slc NQC pair $(X/Z,B+M)$ with data $\overline{M}$ is said to be \emph{with (log) abundant data} if  $\overline{M}$ is abundant (resp. log abundant with respect to $(\overline{X},\overline{B}+\overline{M})$) over $Z$. 
\end{defn}

\begin{lem}\label{lem-convex-combination}
	Notation as above, let $\overline{M}=\sum_i \alpha_i \overline{M}_i$ be the given convex combination and $\mathcal{M}$ be the rational polytope given by the convex combination. Then, there exists a sub-polytope such that each vertex is nef and (log) abundant$/Z$ for every $j$. Moreover, there exists a sub-polytope $\mathcal{M}'$ such that every divisor $\overline{N} \in \mathcal{M}'$ is supported by $\Supp \overline{M}$.
\end{lem}
\begin{proof}
	We only prove the abundance case, since the log abundance case can be argued analogously (using Lemma \ref{lem-nef-abund-1}). Since we are free to pull back $\overline{M}$, by Lemma \ref{lem-nef-abundant-divisor}, we can assume there exists a surjective morphism $g: \overline{X} \to Y$ such that $\overline{M} \sim_\R g^*B$ for some nef and big$/Z$ divisor $B$. By a standard argument of convex polytopes, we can shrink $\mathcal{M}$ to obtain the required combination. The last assertion is obvious.
\end{proof}

\subsection{Iitaka fibrations and parabolic fibrations}\label{subsec-iitaka}
In this subsection, we discuss (invariant) Iitaka fibrations for $\K$-divisors ($\K$=$\R$ or $\Q$). Some parts are taken from \cite{hu3}.

\begin{defn}[Invariant Iitaka fibration, \text{\cite[Definition 2.2]{hu3}}]
	Let $X$ be a normal variety, $f:X \to Z$ be a proper morphism, and $D$ be an $\mathbb{R}$-Cartier divisor on $X$ with $\kappa_{\iota}(X/Z,D) \ge 0$. Pick an $\R$-Cartier divisor $E \ge 0$ such that $D\sim_\R E/Z$. Then there exists a contraction $\phi: X' \to Y$ of smooth varieties such that for all sufficiently large integers $m > 0$, the rational maps $\phi_{m} : X \dashrightarrow Y_m$ given by $f ^*f_*\mathcal{O}_X(\rddown{mE})$ are birationally equivalent to $\phi$, that is, there exists a commutative diagram
	$$
	\xymatrix{
		X \ar@{-->}[d]_{\phi_m}   &  X' \ar[d]^{\phi} \ar[l]_{\pi} &\\
		Y_m  &    Y  \ar@{-->}[l]^{\varphi_m}} 
	$$  
	of rational maps $\phi_m,\varphi_m$ and a contraction $\pi$, where the horizontal maps
	are birational, $\dim Y = \kappa_\iota(X,D)$, and $\kappa(X'/Y,f^*E ) = 0$. Such a fibration is called an \emph{invariant Iitaka fibration} or simply an \emph{Iitaka fibration} of $D$. It is unique up to birational equivalence.
\end{defn}

\begin{rem}[\text{\cite[Lemma 2.3]{hu3}}]
	The definition above is well-defined and independent of the choice of $E$. Note that if $D$ is $\Q$-Cartier, then it is the usual Iitaka fibration.
\end{rem}

We prove the following lemma which is a result of weak toroidal reduction \cite{adk} and equidimensional reduction \cite[Proposition 4.4]{ak}.

\begin{lem}\label{lem-iitaka-fib}
	Let $X$ be a normal variety proper over $Z$, and $D$ be a $\mathbb{K}$-Cartier divisor on $X$ with $\kappa_{\iota}(X/Z,D) \ge 0$. Then there exists a log resolution $\pi:\overline{X} \to X$ of $(X,\Supp D)$ and a projective morphism $f: \overline{X} \to Y$ of smooth quasi-projective varieties over $Z$ such that 
	\begin{enumerate}
		\item $\pi^*D\sim_\K f^*D_Y+ R$ where $D_Y,R \ge 0$ are snc $\K$-divisors and $D_Y$ is big$/Z$;
		
		\item $\kappa(\overline{X}/Y,R^h)=0$ and $R^v$ is very exceptional$/Y$, where $R^h$ (resp. $R^v$) denotes the horizontal (resp. vertical) part.
	\end{enumerate}
\end{lem}

\begin{proof}
	Let $f:\overline{X} \to Y$ be an Iitaka fibration of $0 \le E \sim_\K D$. By making minor changes to the proof of Lemma \ref{lem-iitaka-fib} with an aid of \cite{adk}, we can construct a commutative diagram
	$$
	\xymatrix{
		\overline{X}' \ar[d]_{f'} \ar[r]^{\pi'}   &  \overline{X}\ar[d]^{f}\  &\\
		Y' \ar[r]^{\phi} &    Y } 
	$$
	satisfying (for notations and basic properties of toroidal morphisms, we refer to \cite{hu2}):
	\begin{enumerate}
		\item $\pi'$ and $\phi$ are birational morphisms from quasi-projective varieties. 
		
		\item $f':(\overline{X}',\overline{\Delta}') \to (Y',\Delta_{Y'})$ is toroidal from a quasi-smooth toroidal variety to a smooth variety, and $f'^{-1}(Y' \setminus \Delta_{Y'})$ is smooth.
		
		\item $f'$ is flat.
		
		\item $\pi^{-1}(\Supp \overline{E} \bigcup \Supp \overline{D} \bigcup \ex(\pi)) \bigcup \ex(\pi') \subseteq \overline{\Delta}'$, where $\overline{E}=\pi^*E$, $\overline{D}=\pi^*D$, and $\ex(\pi),\ex(\pi')$ denote the exceptional loci.
	\end{enumerate}
	Replacing $Y$ with $Y'$, $\overline{X}$ with a toroidal resolution of $\overline{X}'$, and replacing $f$ accordingly, we can assume $(\overline{X}, \Supp \overline{D}\bigcup \Supp\overline{E} \bigcup \ex(\pi))$ is log smooth. Moreover, there exists a divisor $D_Y$ supported by $\Delta_Y$ such that $R^v:=\overline{E}^v-f^*D_Y$ is very exceptional$/Y$. Letting $R^h=E^h \ge 0$ we obtain the lemma.
\end{proof}

The term ``parabolic fibration" is after \cite[Section 7]{ambro1}.

\begin{defn}[Parabolic fibration]\label{defn-paprabolic-fib}
	Let $(X/Z,B+M)$ be a g-lc pair 
	and $f:X \to Y$ be a morphism of normal varieties with connected fibres over $Z$. We say $f:(X,B+M) \to Y$ is a \emph{parabolic fibration} if $\kappa_\iota(X/Y,K_X+B+M)=0$. 
\end{defn}

\begin{lem}[\text{cf.\cite[Theorem 1.2]{hu3}}]\label{lem-smooth-modification}
	Let $(X/Z,B+M)$ be a g-sub-lc pair and $\kappa_{\iota} (X/Z,K_X+B+M)\ge 0$. Then, there exists a log resolution $\pi:(\overline{X},\overline{B}+\overline{M}) \to X$, where $K_{\overline{X}}+\overline{B}+\overline{M} =\pi^*(K_X+B+M)$, and a projective morphism $f: \overline{X} \to Y$ of smooth quasi-projective varieties over $Z$ such that 
	\begin{enumerate}
	\item $K_{\overline{X}}+\overline{B}+\overline{M}\sim_\R f^*D_Y+ R$ where $D_Y,R \ge 0$ are snc divisors and $D_Y$ is big$/Z$;
	
	\item $\kappa(\overline{X}/Y,R^h)=0$ and $R^v$ is very exceptional$/Y$, where $R^h$ (resp. $R^v$) denotes the horizontal (resp. vertical) part;
   \end{enumerate}
   Moreover, if $(X,B+M)$ is sub-lc, then $f:(\overline{X},\Xi+\overline{M}) \to Y$ is a parabolic fibration, where $\Xi -\overline{B} \ge 0$ is an exceptional$/X$ divisor such that $(\overline{X},\Xi+\overline{M})$ is g-dlt.
\end{lem}
\begin{proof}
	The proof of Lemma \ref{lem-iitaka-fib} works almost verbatim for the first part of the lemma. The only difference is that, we need to require $$\pi^{-1}(\Supp \overline{E} \bigcup \Supp \overline{B}\bigcup \Supp \overline{M} \bigcup \ex(\pi)) \bigcup \ex(\pi') \subseteq \overline{\Delta}'$$ in Condition (4). For the latter part, it suffices to prove $\kappa_\iota(\overline{X}/Y,K_{\overline{X}}+\Xi+\overline{M})=0$, which comes from the construction of Iitaka fibrations.
\end{proof}

\begin{rem}
	In the above lemma, if $B$ is a $\K$-boundary and the data $\mathbf{M}$ is $\K$-b-Cartier, then we can require $D_{Y}$ and $R$ to be $\K$-Cartier. Moreover, if the data $M=0$, then by \cite[proof of Theorem 1.2]{hu3}, $D_{Y}$ can be written as $K_{Y}+B_{Y}+M_{Y}$ where $(Y,B_{Y}+M_{Y})$ is a g-lc pair with the moduli b-divisor given by a canonical bundle formula (\cite[Section 3.1]{hu2}).
\end{rem}

\section{Finiteness of $B$-representations}\label{sec3}

In this section, we give a proof of Theorem \ref{thmmain-good-lc}, following the strategy of \cite{fujino-gongyo}. We remark that, \cite{haconxu} provides a Hodge theoretic approach, which should work for g-pairs. All divisors in
this section are $\Q$-divisors unless stated otherwise.

\subsection{$\widetilde{B}$-representations}
In this subsection, we rephrase the finiteness of $\widetilde{B}$-pluricanonical representations from \cite{fujino-gongyo} in the language of birational geometry. 

\begin{defn}[$\widetilde{B}$-birational maps]\label{defn-m-B-bir}
	Let $(X,B)$, $(X',B')$ be two proper sub-pairs and $m$ be a positive integer such that $m(K_X+B),m(K_{X'}+B)$ are Cartier. We say that a birational map $\varphi: X \bir X'$ is \emph{$m$-$\widetilde{B}$-birational} if there is a common log resolution,
	$$
	\xymatrix{	
		&  \overline{X}  \ar[dl]_{\alpha}\ar[dr]^{\beta}&\\
		(X, B)  \ar@{-->}[rr]^{\varphi}&   &  (X', B')  &\\
	} 
	$$
	and set $K_X=\alpha_*K_{\overline{X}}$ and $K_{X'}=\beta_*K_{\overline{X}}$ for a fixed canonical section $\omega_{\overline{X}}$, we have that an inclusion
	$$
	 H^0(X',m(K_{X'}+B')) \hookrightarrow H^0(X,m(K_{X}+B))
	$$
	as linear sub-spaces of the function field $k(X)$.
\end{defn}

\begin{rem}\label{rem-m-B-bir}
	The above condition can be interpreted via divisors: Write $K_X+\Delta=\varphi^*(K_{X'}+{B'}):= \alpha_* \beta^*(K_{X'}+B')$, and write $\Delta=B+E-F$, where $E,F\ge 0$ have no common components. Then, for any divisor $D' \in |m(K_{X'}+B')|$, we have $\varphi^*(\frac{1}{m}D) +F-E \ge 0$. 
\end{rem}

\begin{rem}\label{lem-m-B-bir-0}
	Notation as above, suppose $X$ is smooth. Then, $m(K_X+\Delta)$ is Cartier and we have:
		\begin{enumerate}
		\item A birational map $\varphi$ induces an injective linear homomorphism
		$$
		\varphi^*: H^0(X',m(K_{X'}+B')) \hookrightarrow H^0(X,m(K_{X}+\Delta)).
		$$
		So, $\varphi$ is $m$-$\widetilde{B}$-birational if and only if $\mathrm{Im}(\varphi^*) \subseteq H^0(X,m(K_{X}+B))$.
		 
		\item If $\varphi$ is $m$-$\widetilde{B}$-birational and $n|m$ such that $n(K_X+B),n(K_{X'}+B') $ are Cartier, then $\varphi$ is $n$-$\widetilde{B}$-birational.
	\end{enumerate}
\end{rem}

The reader should be careful that, in general, $\varphi^*$ does not send a section $s\in  H^0(X',m(K_{X'}+B'))$ to a section of $H^0(X,m(K_{X}+B))$ directly. The following example is due to Kenta Hashizume.
\begin{exa}
	Let $E$ be an elliptic curve. Pick an automorphism $f: E \to E$ and 
	a prime divisor $P$ such that $f^*P = Q$. Note that for any distinct prime divisors $Q$ and $Q'$ on $E$, $Q$ is not linearly equivalent to $Q'$ (\cite[II, Example 6.10.1]{hartshorne}). 
	So, $H^0(E, 2K_E+P)$ has the unique element corresponding $P$. Consider a klt pair $(E,\frac{1}{2}P)$ and $f$. Clearly $K_E \sim 0$, and one can check that $f: E \to E$ is $2$-$\widetilde{B}$-birational. On the other hand, $Q=f^*P$ is not an element of $H^0(E,2K_E+P)$.
\end{exa}

\begin{lem}\label{lem-m-B-bir}
	Notation as in Definition \ref{defn-m-B-bir} and Remark \ref{rem-m-B-bir}, if $mE \le \mathrm{Fix}(|m(K_X+B+E)|)$, then $\varphi$ is $m$-$\widetilde{B}$-birational.
\end{lem}
\begin{proof}
	For any divisor $mD' \in |m(K_{X'}+B')|$, since $\varphi^*D'+F \ge 0$, we deduce $m(\varphi^*D'+F) \in |m(K_X+B+E)|$.
	By assumption we have $\varphi^*D+F-E \ge 0$. 
\end{proof}

\begin{defn}[$\widetilde{B}$-representations]
	Let $(X,B)$ be a proper sub-pair, and let $m$ be a positive
	integer such that $m(K_X + B)$ is Cartier. An $m$-$\widetilde{B}$-birational map $\varphi: X \bir X \in \widetilde{\Bir}_m(X,B)$ defines a linear automorphism of $H^0(X, m(K_X + B))$. Thus
	we get the group homomorphism
	$$
	\rho_m : \widetilde{\Bir}_m(X,B) \to \mathrm{Aut}_{k}(H^0(X, m(K_X + B))).
	$$
	The homomorphism $\widetilde{\rho}_m$ is called the $\widetilde{B}$-representation
	of $\widetilde{\Bir}_m(X,B)$. We sometimes simply denote $\widetilde{\rho}_m(\varphi)$ by $\varphi^*$
	for $\varphi \in \widetilde{\Bir}_m(X,B)$ if there is no danger of confusion.
\end{defn}

\noindent \textbf{Caution:} Given an $m$-$\widetilde{B}$-birational map $X \bir X$, note that $K_X$ on the left $X$ does not coincide with $K_X$ on the right $X$. But for convenience, we often use the same notation if there is no danger of confusion.

\begin{rem}
	If $(X,B)$ is smooth sub-klt, then the above definition coincides with \cite[Definition 3.1]{fujino-gongyo}. 
	
	Indeed, given a birational automorphism $\varphi: X \bir X$, since $X$ is smooth, the relative canonical divisor of $\varphi|_U:U \to U'$, where $U$ is the regular locus of $\varphi$, is $K_\varphi$, which is exceptional$/X$ and extends to a Cartier divisor on $X$. By an easy calculation, we have 
	$$
	\Delta=B+E-F
	$$
	where $\Delta|_U:=(\varphi|_U)^*(B|_{U'}) -K_\varphi$ is uniquely defined, and $E,F$ defined in Remark \ref{rem-m-B-bir}. So, for any nonzero meromorphic $m$-ple $n$-form $\omega$ which vanishes along $-mB^{<0}$ and
	has poles at most $mB^{>0}$, $(\varphi|_U)^*(\omega|_{U'})$ is a meromorphic $m$-ple $n$-form whose extension $\widetilde{\omega}$ on $X$ vanishes along $-m\Delta^{<0}$ and
	has poles at most $m\Delta^{>0}$. That is, $\varphi$ induces an injective homomorphism
	$$
	\varphi^*: H^0(X',m(K_{X'}+B)) \hookrightarrow H^0(X,m(K_{X}+\Delta))
	$$
	in the sense of \cite[Definition 3.1]{fujino-gongyo}. Finally, we note that, $\widetilde{\omega}$ is a meromorphic $m$-ple $n$-form which vanishes along $-mB^{<0}$ and
	has poles at most $mB^{>0}$ if and only if $(\widetilde{\omega})+F-E \ge 0$.
\end{rem}

We cite the following theorem from \cite{fujino-gongyo}.
\begin{thm}[\text{\cite[Theorem 3.9]{fujino-gongyo}}]\label{thm-B'-rep}
	Let $(X, B)$ be a projective log smooth sub-klt pair, and $m(K_X + B)$ is
	Cartier where $m$ is a positive integer. Then $\widetilde{\rho}_m( \widetilde{\Bir}_m(X, B))$ is a finite group.
\end{thm}

\begin{rem}
	By the Lefschetz principle and flat base change, Theorem \ref{thm-B'-rep} also holds for varieties over an algebraically closed field of characteristic zero. 
\end{rem}

\subsection{$B$-representations}
We start our study with \emph{$B$-birational maps} and \emph{$B$-representations}. These notions were firstly introduced in \cite{fujino1} and thoroughly studied in \cite{fujino-gongyo}\cite{haconxu}, which plays a crucial role to arrange inductive proofs in higher dimensional algebraic geometry. To carry out further research into lc pairs, we generalise these notions to g-sub-pairs.

\begin{defn}[$B$-birational maps]\label{defn-B-bir}
	Two g-sub-pairs $(X/S,B+M)$ and $(X'/S,B'+M')$ are \emph{$B$-birational over $S$} if there is a common log resolution over $S$
		$$
	\xymatrix{	
		&  (\overline{X},\overline{B}+\overline{M})  \ar[dl]_{\alpha}\ar[dr]^{\beta}&\\
		(X, B+M) \ar[dr]_{} \ar@{-->}[rr]^{\varphi}&   &  (X', B'+M') \ar[dl]_{} &\\
		&   S &  &\\
	} 
	$$
	and write 
	$$K_{\overline{X}}+\overline{B}+ \overline{M} = \alpha^* (K_X+B+M);~ K_{\overline{X}}+\overline{B}'+ \overline{M}' = \beta^*(K_{X'}+B'+M'), $$
	where we set $K_X=\alpha_*K_{\overline{X}}$ and $K_{X'}=\beta_*K_{\overline{X}}$,
	then $\overline{B}=\overline{B}'$ and $\overline{M} = \overline{M}'$. In this case, the rational map $\varphi:=(\beta \circ \alpha^{-1}): X \dashrightarrow X'$ is a \emph{$B$-birational map} (cf.\cite[Definition 1.5]{fujino1}). We omit $S$ when it is a point. 
\end{defn}

\begin{exa}[\text{Quadratic transformation, cf.\cite[Example 2.13]{fujino-gongyo}}]
	Let $X = \PP^2$ and let $\Delta$	be the union of three general lines on $\PP^2$. Let $\alpha : \overline{X} \to X$ be the blow-up at the three intersection points of $\Delta$ and let $\beta : \overline{X} \to X$ be
	the blow-down of the strict transform of $\Delta$ on $\overline{X}$. Then we obtain the
	quadratic transformation $\varphi$. In this situation, it
	is easy to see that
	$$
	\alpha^*(K_X + \Delta) = K_{\overline{X}} + \overline{\Delta} = \beta^*(K_X + \Delta).
	$$
	\begin{enumerate}
		\item Set $B=\Delta$ and $M=0$. Then, $\varphi$ is $B$-birational.
		
		\item Set $B=\frac{1}{2}\Delta$, $M=\frac{1}{2}\Delta$, and $\mathbf{M}$ descends to $X$. Then, $\varphi$ is NOT $B$-birational.
		
		\item Set $B=\frac{1}{2}\Delta$, $M=\frac{1}{2}\Delta$, and $\mathbf{M}$ descends to $\overline{X}$ with $\mathbf{M}_{\overline{X}}=\frac{1}{2}\overline{\Delta}$. Then, $\varphi$ is $B$-birational.
	\end{enumerate}
\end{exa}

\begin{defn}[$B$-representations]\label{defn-B-rep}
	Let $(X,B+M)$ be a proper g-sub-pair, and let $m$ be a positive
	integer such  $m(K_X + B+M)$ is Cartier. A $B$-birational map $\sigma: X \bir X \in 
	\Bir(X, B,M)$ defines a linear automorphism of $H^0(X, m(K_X + B+M))$. Thus
	we get the group homomorphism
	$$
	\rho_m : \Bir(X, B,M) \to \mathrm{Aut}_{k}(H^0(X, m(K_X + B+M))).
	$$
	The homomorphism $\rho_m$ is called a $B$-representation for $(X, B+M)$. We sometimes simply denote $\rho_m(\varphi)$ by $\varphi^*$
	for $\varphi \in \Bir(X, B,M)$ if there is no danger of confusion. .
\end{defn}

\noindent \textbf{Caution:} Given a $B$-birational map $X \bir X$, note that $K_X$ on the left $X$ does not coincide with $K_X$ on the right $X$. But for convenience, we often use the same notation if there is no danger of confusion.\\

The following lemma is elementary. So we omit the proof.
\begin{lem}\label{lem-B-rep}
	Notation as in Definitions \ref{defn-B-bir} and \ref{defn-B-rep}, we have the followings.
	\begin{enumerate}
		\item Let $\pi:X' \to X$ be a birational model on which $K_{X'}+B'+M'=\pi^*(K_X+B+M)$. Then, $\Bir(X',B',M')=\Bir(X,B,M)$.
		
		\item If $(X,B+M)$ is log smooth sub-klt (as sub-pairs), then $$ \Bir(X,B,M) \subset \widetilde{\Bir}_m (X,B+M)$$ for every $m$. In particular, $\rho_m(\Bir(X,B,M))$ are finite groups.
	\end{enumerate}
\end{lem}

\begin{lem}\label{lem-slc-contraction}
	Let $(X,B+M)$ be a g-lc pair with data $\overline{M}$, and $\pi:(\overline{X},\overline{B}+\overline{M}) \to X$ be a log resolution. Then, the restriction $\pi|_{\overline{S}}: \overline{S} \to S$ has connected fibres, i.e. $\pi|_{\overline{S},*} \mathcal{O}_{\overline{S}}=\mathcal{O}_S$, where $\overline{S}= \overline{B}^{=1}, S=B^{=1}$. Moreover, if $X$ is quasi-projective $\Q$-factorial klt, then there exists a datum $\overline{M}' \sim_\K \overline{M}$ such that the induced g-pair $(S,B_S+M_S')$ is well-defined and g-slc.
\end{lem}
\begin{proof}
	Set $E=\lceil-\overline{B}^{<0}\rceil$. Then the contraction follows from Kawamata-Viehweg vanishing $$R^1\alpha_* \mathcal{O}_{\overline{X}}(m(K_{\overline{X}}+\overline{B}+\overline{M})+E-\overline{S}) =0.$$
    
    Suppose $X$ is $\Q$-factorial klt. Since $(X,B)$ is lc, the scheme $S$ is $S_2$ and normal crossing in codimension one. Suppose $\overline{M}$ does not contain any strata in its support. It is obvious that $(S_i,B_{S_i}+M_{S_i})$ is g-lc for every irreducible component $S_i$ of $S$. It remains to check Condition (2) of Definition \ref{defn-g-slc-pair}.
    
    Let $T$ be a component of the conductor on $S$. Due to the next lemma, a stratum $\overline{T}$ of $\overline{X}$ dominant over $T$ is either birational to $T$ or $\overline{T} \to T$ has a general fibre $\PP^1$ and two horizontal strata. By an inductive argument, it is sufficient to show that, for such a morphism $\overline{T} \to T$ with general fibres $\PP^1$, two horizontal strata $W,W' \subset \overline{T}$ birational to $T$, we have $\mathbf{M}|_{W}= \mathbf{M}|_{W'}$ as b-divisors. 
    
    To this end, since $X$ is $\Q$-factorial, we write $D=K_X+B$, $\pi^*D=K_{\overline{X}}+\overline{B}-E$ and $\overline{M}'=\pi^*M'-E$ so that $\overline{M}' \sim_\K \overline{M}$ and that both $D,M'$ contain no components of $S$ or the conductor of $S$. This can be done since $X$ is quasi-projective. Observe that $\overline{M}'|_{\overline{T}} \equiv 0$ over the generic point of $T$, hence $\overline{M}'|_{\overline{T}} \sim_\K 0$ over the generic point of $T$, which in turn implies that $\overline{M}'|_{\overline{T}}$ is nef and abundant$/T$. Moreover, since $K_{\overline{T}}+B_{\overline{T}} + \overline{M}'|_{\overline{T}} =0/T$, we deduce $\overline{M}'|_{\overline{T}} = 0$ over the generic point of $T$. By Lemma \ref{lem-nef-abundant-divisor}(2) and replacing $\overline{T} \to T$, we can assume $\overline{M}'|_{\overline{T}} = 0/T$, hence we achieved the lemma.
\end{proof}

\begin{lem}
	Let $(X,B+M)$ be a g-lc pair, where $X$ is $\Q$-factorial klt and $\pi:(\overline{X},\overline{B}+\overline{M}) \to X$ be a log resolution. Then, for any component $T$ of the conductor on $S$, a general fibre of $\pi|_{\overline{S}}^{-1}T \to T$ is either a single point or a connected chain of $\PP^1$.
\end{lem}
\begin{proof}
	Cutting by general hyperplanes, we can assume $T$ is a point. Hence the lemma is obvious.
\end{proof}

As a corollary, we obtain the following lemma.
\begin{lem}[\text{cf.\cite[Remark 2.15]{fujino-gongyo}}]\label{lem-res}
	Given a proper log smooth g-sub-dlt pair $(X,B+M)$ and $\varphi \in \Bir (X,B,M)$, let $S= \rddown{B}$. Then, we have an induced automorphism
	$$
	\varphi^*:H^0(S,\mathcal{O}_S(m(K_S+B_S+M_S)) \overset{\sim}{\longrightarrow} H^0(S,\mathcal{O}_S(m(K_S+B_S+M_S))
	$$
	where $K_S+B_S+M_S=(K_X+B+M)|_S$ given by adjunction, and $m$ is a non-negative
	integer such that $m(K_X + B+M)$ is Cartier.
\end{lem}

The next two lemmas are analogous to \cite[Lemma 4.9, Claim ($A_n$) and Claim ($B_n$)]{fujino1}. Note that the original proofs work verbatim for g-pairs.

\begin{lem}\label{lem-stratum-1}
	Let $(X ,B+M)$ and $(X',B'+M')$ be log smooth g-sub-dlt pairs and $\pi : X' \to X$ a $B$-birational morphism. If we are given a stratum $W \subsetneq X$, then there is a stratum $W' \subsetneq X'$ such that $\pi|_{W'}:W' \to  W$ is $B$-birational.
\end{lem}
\begin{proof}
	By induction on the dimension of $X$.
\end{proof}

\begin{lem}\label{lem-stratum-2}
	Let $(X ,B+M)$ and $(X',B'+M')$ be proper log smooth g-sub-dlt pairs and $\pi : X' \to X$ a $B$-birational morphism. Let $W' \subsetneq X'$ be a stratum. If $W' \to W:=\pi(W' )$ is not $B$-birational, then there is a stratum $V' \subsetneq W'$ such that $V' \to  W$ is a $B$-birational morphism and the inclusion $V' \hookrightarrow  W'$ induces the isomorphism 
	$$
	H^0(W',m(K_{W'} +B_{W'}+M_{W'}) ) \overset{\sim}{\longrightarrow}  H^0(V',m(K_{V'} +B_{V'}+M_{V'}) ),
	$$
	where $m$ is a positive integer such that $m(K_X + B+M)$ is Cartier.
\end{lem}
\begin{proof}
	Suppose Lemma \ref{lem-stratum-2} holds for dimension $\le d-1$ where $d=\dim X$. Blowing-up $W'$ we can assume it is a divisor. By repeatedly blowing-up strata of $X$, there exists a a sequence of blow-ups $X_0 \to X_1 \to \cdots \to X_k = X$ whose centres are the centres associated to the valuation $W'$ such that the rational map $X' \bir X_0$ is an isomorphism at the generic point of $W'$. Hence, by induction, $V' \bir W_0$ is $B$-birational. Replacing $X'$ we can assume $\pi$ is a sequence of blow-ups of strata. By induction on $k$ we can assume $\pi$ is a single blow-up, hence we conclude the lemma by the inductive assumption in dimension $d-1$.
\end{proof}

Combining Lemmas \ref{lem-stratum-1} and \ref{lem-stratum-2}, we obtain the next lemma.

\begin{lem}[\text{cf.\cite[Lemma 2.16]{fujino-gongyo}}]\label{lem-induction}
	Let $\varphi : (X, B+M) \bir (X', B'+M')$ be a $B$-birational map between proper log smooth g-sub-dlt pairs, and $S$ be a stratum of $(X, B+M)$ with
	$K_S+B_S+M_S=(K_X+B+M)|_S$. We take a suitable common log resolution as
		$$
	\xymatrix{	
		&  (\overline{X},\overline{B}+\overline{M})  \ar[dl]_{\alpha}\ar[dr]^{\beta}&\\
		(X, B+M)  \ar@{-->}[rr]^{\varphi}&   &  (X', B'+M')  &\\
	    } 
	$$
	Then we can find a stratum $V$ of $(X, B+M)$ contained in $S$ with $K_V + B_V+M_V = (K_X +B+M)|_V$, a stratum $\overline{V}$ of $(\overline{X},\overline{B}+\overline{M})$ with $K_{\overline{V}} +B_{\overline{V}}+M_{\overline{V}} = (K_{\overline{X}} +\overline{B} +\overline{M})|_{\overline{V}}$,
	and a stratum $V'$ of $(X', B'+M')$ with $K_{V'} + B_{V'}+M_{V'} = (K_{X'} + B'+M')|_{V'}$ such
	that the following conditions hold:
	\begin{enumerate}
		\item $\alpha|_{\overline{V}}$ and $\beta|_{\overline{V}}$ are $B$-birational morphisms. Therefore, $\varphi|_V=(\beta|_{\overline{V}} ) \circ (\alpha|_{\overline{V}} )^{-1} : (V, B_V+M_V ) \bir (V', B_{V'}+M_{V'})$ is a $B$-birational map.
		
		\item $H^0(S, m(K_S + B_S+M_S)) \simeq H^0(V, m(K_V + B_V+M_V ))$ by the natural restriction map where $m$ is a positive integer such that $m(K_X + B+M)$ is Cartier.
	\end{enumerate}
\end{lem}

Now we turn to proving the finiteness of $B$-representations. We remark that neither the positivity (of log canonical divisors or data) nor the effectiveness of boundaries is needed. Before we state the finiteness theorem, we introduce the following definition which is crucial to the arguments in Subsections \ref{subsec-vertical} and \ref{subsec-horizontal}.

\begin{defn}[General data]\label{defn-g-data}
	Let $(X/Z, B+M)$ be a g-sub-lc pair or a g-slc pair with data $\overline{M}$ and $\overline{X} \to X$. We say $\overline{M}$ is a \emph{general datum} if
	\begin{enumerate}
		\item $(\overline{X},\overline{B}+\overline{M})$ is a quasi-projective log smooth resolution;
		
		\item The support of $\overline{M}$ contains no strata;
		
		\item  $(\overline{X},\overline{B} - \epsilon \overline{B}^{=1} +\overline{M})$ is log smooth and sub-klt as a sub-pair for every number $\epsilon >0$.
	\end{enumerate}
    If further $(X,B+M)$ is g-dlt or g-sdlt, then we say a general datum $\overline{M}$ is \emph{very general} provided that
    \begin{enumerate}
    	\item[(4)] $a(E,X,B+M)>0$ for every exceptional divisor $E$ on $\overline{X}$. 
    \end{enumerate}
\end{defn}

\begin{rem}\label{rem-g-data}
	By saying $(X, B+M)$ is g-sub-lc pair or a g-slc pair \emph{with general data} $\overline{M}$, we mean $\overline{M}$ is chosen to be sufficiently general on a sufficiently high resolution $\overline{X}$.
	\begin{enumerate}
		\item By the proof of Lemma \ref{lem-g-sdlt}(1), given a g-sub-lc pair or a g-slc pair $(X,B+M)$ with data $\overline{M}$ such that $m\overline{M}$ is Cartier for some positive integer $m \ge 2$, we can always find a general data $\overline{M'}$ with $m\overline{M}' \sim m\overline{M}$ and the coefficients of $\overline{M} \le \frac{1}{m}$.
		
		\item If $(X,B+M)$ is further quasi-projective g-sdlt, then by Lemmas \ref{lem-g-dlt} and \ref{dlt-resolution}, we can always find a very general data $\overline{M'}$ with $m\overline{M}' \sim m\overline{M}$ and the coefficients of $\overline{M} \le \frac{1}{m}$.
	\end{enumerate}
\end{rem}

\begin{thm}[\text{cf.\cite[Theorem 3.15]{fujino-gongyo}}]\label{thm-good-lc}
	Let $(X, B+M)$ be a proper g-sub-lc pair with general data $\overline{M}$. Let $m \ge 2$ be an integer such that $m(K_X+B+M), m\overline{M}$ are Cartier.
	Then $\rho_m(\Bir(X, B,M))$ is a finite group. 
\end{thm}

Subsections \ref{subsec-vertical} and \ref{subsec-horizontal}  will be devoted to proving the above theorem.

\subsection{Proof of Theorem \ref{thm-good-lc}: the vertical case}\label{subsec-vertical}
    We begin with the g-sub-klt case.
\begin{lem}[\text{cf.\cite[Corollary 3.10]{fujino-gongyo}}]\label{lem-klt}
	Let $(X, B+M)$ be a proper g-sub-klt pair with general data $\overline{M}$. Let $m \ge 2$ be a positive integer such that $m(K_X + B+M),m\overline{M}$ are Cartier.
	Then $\rho_m(\Bir(X, B,M))$ is a finite group.
\end{lem}    
\begin{proof}
	The lemma follows from Lemma \ref{lem-B-rep}(2) and Theorem \ref{thm-B'-rep}.
\end{proof}
    
\begin{prop}[\text{cf.\cite[Theorem 3.11]{fujino-gongyo}}]\label{prop-good-lc}
	Let $(X, B+M)$ be an irreducible projective g-sub-lc pair with general data $\overline{M}$, and $f: X' \to Y$ be an Iitaka fibration of $K_X+B+M$. Suppose that all g-sub-lc centres are vertical over $Y$. Let $m \ge 2$ be a positive integer such that $m(K_X + B+M),m\overline{M}$ are Cartier. Then $\rho_m(\Bir(X, B,M))$ is a finite group.
\end{prop}
\begin{proof}
	By Lemma \ref{lem-iitaka-fib} and replacing $X$ with $\overline{X}$, we can assume $(X,B+M)$ is log smooth projective and
	\begin{enumerate}
		\item $K_{X}+B+M \sim_\Q f^*D_Y+ R$ where $D_Y,R \ge 0$ are snc divisors and $D_Y$ is big;
		
		\item $\kappa(X/Y,R^h)=0$ and $R^v$ is very exceptional$/Y$, where $R^h$ (resp. $R^v$) denotes the horizontal (resp. vertical) part.
	\end{enumerate}
	By Lemma \ref{lem-klt}, we can assume $B^{=1}\neq 0$. Moreover, since $B^{=1}$ is vertical, there exists an effective Cartier divisor $S_Y$, such that $ f^*S_Y \ge S:=B^{=1} $. Since $D_Y$ is big, there is an integer $m'$ such that $|m' D_Y -S_Y| \neq 0$ and 
	$$m'(K_X +B+M) \sim f^*m' D_Y+m'R,$$ 
	hence $|m'(K_X+B+M)-S| \neq 0$. Since $\varphi^*S \ge S$, by Lemma \ref{lem-m-B-bir}, $\varphi \in \widetilde{\Bir}_{m'}(X,B- \frac{1}{m'} S+M)$ which in turn implies that we have a natural inclusion
    $$
    \Bir(X,B,M) \subset \widetilde{\Bir}_{m'}(X,B- \frac{1}{m'} S+M)
    $$
    By Theorem \ref{thm-B'-rep}, $\widetilde{\rho}_{m'} (\widetilde{\Bir}_{m'}(X,B- \frac{1}{m'} S +M))$ is a finite group, hence $\varrho_{m'}(\Bir(X,B,M))$ is a finite group, where $\varrho_{m'}:\Bir(X,B,M) \to \mathrm{Aut}(H^0(X,m'(K_X+B+M)-S))$.
    
    Letting $a=|\varrho_{m'}(\Bir(X,B,M))| < \infty$, we can find a $\Bir(X, B,M)$-invariant non-zero section $$s \in H^0(X, a(m'(K_X +B+M)-S)).$$  
    By using $s$, we have a natural inclusion.
    \begin{equation*}
    	H^0(X,m(K_X+B+M)) \overset{\cdot s}{\longrightarrow}H^0(X,(m+am')((K_X+B+M) -aS)) 
    \end{equation*}
    Since $\varphi^*$ is nothing but an identity on the function field, we see that $\Bir(X,B,M)$ acts on both linear spaces compatibly, hence $\varrho_{m+am'}$ factors through $\rho_m$. Because $\varrho_{m+am'}(\Bir(X,B,M)) \subset \widetilde{\rho}_{m+am'}(\widetilde{\Bir}_{m+am'}(X,B-\frac{a}{m+am'}S+ M))$, we conclude that $\rho_m(\Bir(X, B,M))$ is a finite group by Theorem \ref{thm-B'-rep}.
\end{proof}

\begin{cor}[\text{cf.\cite[Corollary 3.13]{fujino-gongyo}}]\label{cor-B-rep}
	Let $(X, B+M)$ be a proper g-sub-lc pair with general data such that $K_X +B+M$
	is big. Then $\Bir(X, B,M)$ is a finite group. 
\end{cor}
\begin{proof}
	By Proposition \ref{prop-good-lc}, $\rho_m(\Bir(X, B,M))$ is a finite group for $m$ sufficiently divisible. Note that $\rho_m$ induces the group homomorphism
	$$ 
	\Bir(X,B,M) \to \mathrm{Aut}(\PP(H^0(X, m(K_X + B+M)))),
	$$
	which is injective since the induced map is birational. Hence the conclusion follows.
\end{proof}

\subsection{Proof of Theorem \ref{thm-good-lc}: the horizontal case}\label{subsec-horizontal}
It remains to verify the case that there is a horizontal g-sub-lc centre, which completes the whole proof. The strategy of the proof is borrowed from \cite[Theorem 3.15]{fujino-gongyo}.

\begin{proof}[Proof of Theorem \ref{thm-good-lc}]
	It is sufficient to prove the theorem on each connected component, so we assume $X$ is irreducible. Replacing $(X, B+M)$ we may assume it is log smooth g-sub-dlt. By Lemma \ref{lem-iitaka-fib} and replacing $X$ with $\overline{X}$, we can assume $(X,B+M)$ is log smooth projective and $f:X \to Y$ is an Iitaka fibration of $K_X+B+M$ satisfying the conditions listed. By Lemma \ref{lem-klt}, we can assume $T:=B^{=1}\neq 0$. By Proposition \ref{prop-good-lc}, we can assume $T$ is not vertical over $Y$. 
	
	\noindent \textbf{Effective case.} First we assume $T^h \le mR^h$. Since there exists an effective Cartier divisor $S_Y$ such that $f^*S_Y \ge T^v$, there is an integer $m'$ such that $|m'(K_X+B+M)-T| \neq 0$. The argument for the vertical case works without any changes here.
	
	\noindent \textbf{Non-effective case.} Now we assume $T^h \nleq mR^h$. Hence, since $m(K_X + B+M) \sim_\Q f^*mD_Y+ mR$, where $\kappa(X/Y,R^h)=0$, we see $H^0(X,m(K_X+B+M)-T)=0$. Thus, the restriction map
	\begin{equation*}\tag{$\dag$}
		H^0(X,m(K_X+B+M)) \to H^0(T,m(K_T+B_T+M_T))
	\end{equation*}
	is injective, where $K_T+B_T+M_T=(K_X+B+M)|_T$. Let $(V_i, B_{V_i}+M_{V_i})$ be the disjoint
	union of all the $i$-dimensional strata of $(X, B+M)$ for $0 \le i \le n-1$. By induction on dimension,  $\rho_m(\Bir(V_i, B_{V_i},M_{V_i}))$ is a finite group. Put $k_i = |\rho_m(\Bir(V_i, B_{V_i},M_{V_i}))| < \infty$. Let $l$ be the least common multiple of $k_i$, $T = \sum_j T_j$ be the irreducible decomposition and $\varphi \in \Bir(V_i, B_{V_i},M_{V_i})$. By Lemma \ref{lem-induction}, for every $T_j$, there are strata $S^i_j$ 
    $$
    \begin{array}{cclclcc} 
    	X &\stackrel{\pi}{\dashrightarrow} & X &\stackrel{\pi}{\dashrightarrow} & X &\stackrel{\pi}{\dashrightarrow} & \cdots\\
    	\cup & & \cup & & \cup \\
    	S^0_j &  & S^1_j  &  & S^2_j 
    \end{array}
    $$
    such that $S^i_j \dashrightarrow S^{i+1}_j$ is $B$-birational and 
    $$
    H^0(T_j, m(K_{T_j} + B_{T_j} +M_{T_j} )) =H^0(S^i_j, m(K_{S^i_j} + B_{S^i_j} +M_{S^i_j} )) 
    $$
    where $K_{S^i_j} + B_{S^i_j} +M_{S^i_j}=(K_X+B+M)|_{S_j^i}$. Since there are only finitely many strata, we can find $p_j < q_j$ such that $S^{p_j}_j = S_j^{q_j}$.
    Therefore, $\varphi$ induces an $B$-birational map $\varphi_j:\coprod_{p_j\le r < q_j} S_j^r \dashrightarrow \coprod_{p_j\le r < q_j} S_j^r $ for every $j$, which in turn implies $(\varphi_j^*)^l=\mathrm{id}$. In particular, $(\varphi_j^*)^l|_{S^{p_j}_j}=\mathrm{id}$. We have an embedding
    $$
     H^0(T, m(K_{T} + B_{T} +M_{T} )) \subset \bigoplus_j H^0(S^{p_j}_j, m(K_{S^{p_j}_j} + B_{S^{p_j}_j} +M_{S^{p_j}_j} )) 
    $$
    By Lemma \ref{lem-res}, there is a commutative diagram.
    $$
    \xymatrix{
    	H^0(T, m(K_{T} + B_{T} +M_{T} )) \ar[d]^{(\varphi^*)^l=\mathrm{id}} \ar@{^{(}->}[r]^{}   &  \bigoplus_j H^0(S^{p_j}_j, m(K_{S^{p_j}_j} + B_{S^{p_j}_j} +M_{S^{p_j}_j} )) \ar[d]^{(\varphi^*)^l=\mathrm{id}}\  &\\
    	H^0(T, m(K_{T} + B_{T} +M_{T} )) \ar@{^{(}->}[r]^{} &    \bigoplus_j H^0(S^{p_j}_j, m(K_{S^{p_j}_j} + B_{S^{p_j}_j} +M_{S^{p_j}_j} ))  } 
    $$
    By the inclusion ($\dag$), we deduce that every $\varphi \in \Bir(V_i, B_{V_i},M_{V_i})$ acts on $H^0(X,m(K_X+B+M))$ periodically of uniformly bounded order. Therefore, $\rho_m(\Bir(X, B,M))$ is a finite group by Burnside's Theorem (\cite[Theorem 14.9]{ueno}).
\end{proof}


\subsection{Relative $B$-representations}
We generalise the theory of $B$-representations to the relative setting for more applications.

\begin{defn}[Relative $B$-representations]\label{defn-rel-B-rep}
	Let $(X/Z,B+M)$ be a g-sub-pair such that all connected components of $X$ are mapped onto the same locus of $Z$ with generic point $\eta$. Let $m$ be a positive
	integer such that $m(K_X + B+M)$ is Cartier. A $B$-birational map $\varphi \in 
	\Bir(X/Z, B,M)$ over $Z$ defines a linear automorphism of $H^0(X_{\overline{\eta}}, m(K_{X_{\overline{\eta}}} + B_{\overline{\eta}}+M_{\overline{\eta}}))$, where $X_{\overline{\eta}}$ is the geometric generic fibre of $X$ and $ B_{\overline{\eta}},M_{\overline{\eta}}$ are the restriction divisors. Thus
	we get the group homomorphism
	$$
	\rho_m : \Bir(X/Z, B,M) \to \mathrm{Aut}_{\overline{k(\eta)}}(H^0(X_{\overline{\eta}}, m(K_{X_{\overline{\eta}}} + B_{\overline{\eta}}+M_{\overline{\eta}}))).
	$$
	where $\overline{k(\eta)}$ denotes the algebraic closure of the residue field at $\eta$.
	The homomorphism $\rho_m$ is called a relative $B$-representation for $(X/Z, B+M)$. We sometimes simply denote $\rho_m(\varphi)$ by $\varphi^*$
	for $\varphi \in \Bir(X/Z, B,M)$ if there is no danger of confusion. 
\end{defn}

\begin{lem}\label{rel-lem-rep}
	Notation as above, we have $$\Bir(X/Z,B,M) \subseteq \Bir(X_{\overline{\eta}}, B_{\overline{\eta}},M_{\overline{\eta}}). $$
	In particular, we have $$\rho_m(\Bir(X/Z,B,M)) \subseteq \rho_m(\Bir(X_{\overline{\eta}}, B_{\overline{\eta}},M_{\overline{\eta}})).$$ 
\end{lem}
\begin{proof}
	This is because a $B$-birational map over $Z$ is uniquely determined by its restriction to the geometric generic fibre.
\end{proof}

\begin{defn}[Relative $B$-representations: general case]\label{defn-rel-B-rep-2}
	Let $(X/Z,B+M)$ be a g-sub-pair. Write $X= \coprod_i X_i$ where $X_i$ is the disjoint union of all connected components of $X$ mapped onto the same locus of $Z$ with generic point $\eta_i$. We define $\Bir(X/Z, B,M)=\prod_i \Bir(X_i/Z, B,M) $ and $\rho_m(\Bir(X/Z, B,M))=\prod_i \rho_m(\Bir(X_i/Z, B,M))$.
\end{defn}

\begin{thm}\label{thm-finiteness}
	Let $(X/Z, B+M)$ be a g-sub-lc pair with general data $\overline{M}$. Let $m \ge 2$ be an integer such that $m(K_X + B+M), m\overline{M}$ are Cartier.
	Then $\rho_m(\Bir(X/Z, B,M))$ is a finite group.
\end{thm}

\begin{proof}
	Combine Lemma \ref{rel-lem-rep} and Theorem \ref{thm-good-lc}.
\end{proof}

\begin{cor}[\text{cf.\cite[Corollary 3.13]{fujino-gongyo}}]
	Let $(X/Z, B+M)$ be a g-sub-lc pair with general data such that $K_X +B+M$
	is big$/Z$. Then $\Bir(X/Z, B,M)$ is a finite group.	
\end{cor}
\begin{proof}
	Combine Lemma \ref{rel-lem-rep} and Corollary \ref{cor-B-rep}.
\end{proof}

\section{Abundance for generalised pairs}\label{sec4}
In this section we prove abundance theorems for g-slc pairs and for g-dlt pairs. We follow the strategy of \cite{fujino1} is to glue admissible sections on g-dlt pairs.

\subsection{Global generation of admissible sections}\label{sec4-1}
The aim of this subsection is to prove Theorem \ref{thm-adm-gen}, following the strategy of \cite[Section 4]{fujino1}. We remark that the arguments in \cite[Section 4]{fujino1} can be generalised to the relative setting without many changes. Also note that \cite{haconxu-lcc}\cite{haconxu} provide a different approach, based on Koll\'ar's glueing theory, which might also be able to work for g-pairs. For the reader's convenience, we present a self-contained proof here.

Throughout this subsection, all varieties are assumed to be over an integral affine variety $\Spec A$, and all divisors are $\Q$-divisors, unless stated otherwise. We say a variety is \emph{proper} (resp. \emph{projective}) if it is proper (resp. projective) over $\Spec A$. We simply write $\Bir(X, B,M)$ for $\Bir(X/Z, B,M)$ if there is no confusion.

We believe it is not difficult to generalise all results in this subsection over an arbitrary variety $Z$.

\begin{defn}[Pre-admissible and admissible, \text{cf.\cite[Definition 4.1]{fujino1}}]\label{defn-adim-sec}
	Let $(X,B+M)$ be a projective g-sdlt pair with very general data $\overline{M}$, $X=\bigcup_i X_i$ be the irreducible decomposition, and let $m$ be a sufficiently large
	and divisible integer. We define admissible and pre-admissible sections inductively on
	dimension:
	\begin{itemize}
		\item $s\in H^0(X,m(K_X+B+M))$ is pre-admissible if the restriction $s|_{S} \in H^0(S,m(K_S+B_S+M_S))$ is admissible, where $S=\rddown{B}$ and $K_S+B_S+M_S=(K_X+B+M)|_S$;
		
		\item $s\in H^0(X,m(K_X+B+M))$ is admissible if it is pre-admissible and $\varphi^*(s|_{X_j})=(s|_{X_i})$ for every $B$-birational map $\varphi:(X_i,B_i+M_i) \dashrightarrow (X_j,B_j+M_j)$ over $\Spec A$ and for every $i,j$.
	\end{itemize}
    We define linear subspaces of $H^0(X,m(K_X+B+M))$ as follows:
$$
\mathcal{PA}(X,m(K_X+B+M)):= \{\text{$s$ is pre-admissible}\}
$$
and
$$
\mathcal{A}(X,m(K_X+B+M)):= \{\text{$s$ is admissible}\}
$$
\end{defn}

We note that the above definition is well-defined by Lemma \ref{lem-g-sdlt}(3).

\begin{rem}\label{rem-B-bir}
	Notation as above, for every $X_i$, we denote by $X_{i,\overline{\eta}}$ the geometric generic fibre over $\Spec A$. Since $\varphi$ and $s$ are uniquely determined by their restrictions to the geometric generic fibre, it follows that, $s$ is admissible if and only if it is pre-admissible and $\varphi_{\overline{\eta}}^*(s|_{X_{j,\overline{\eta}}})=(s|_{X_{i,\overline{\eta}}})$ for every $B$-birational map $\varphi :(X_{i},B_{i}+M_{i}) \dashrightarrow (X_{j},B_{j}+M_{j})$ over $\Spec A$. In particular, such $s|_{X_{i}}$ is $\Bir(X_{i},B_{i},M_{i})$-invariant for every $i$.		
\end{rem}

The following lemma is direct by definition.
\begin{lem}[\text{cf.\cite[Lemma 4.2]{fujino1}}]\label{lem-descend}
	Let $(X,B+M)$ be a proper g-slc pair with general data $\overline{M}$, $\nu : X^\nu \to X$ be the normalization, and
	$K_{X^\nu} +B^\nu +M^\nu:= \nu^*(K_X +B+M)$. Let $\pi: (X',B'+M') \to (X^\nu,B^\nu +M^\nu)$ be a proper birational morphism
	such that $(X',B'+M')$ is g-dlt with $\overline{M}$ very general, and $K_{X'}+B'+M'= \pi^*(K_{X^\nu} +B^\nu +M^\nu)$. Then any section $s\in \mathcal{PA}(X,m(K_{X'}+B'+M'))$ descends
	to a section on $(X,B+M)$. Moreover, we have:
	\begin{enumerate}
		\item If $(X,B+M)$ is projective g-sdlt with $\overline{M}$ very general, then $s$ descends to $\mathcal{PA}(X,m(K_X+B+M))$.
		
		\item If further $s \in \mathcal{A}(X,m(K_{X'}+B'+M'))$, then $s$ descends to $\mathcal{A}(X,m(K_X+B+M))$.
	\end{enumerate}
\end{lem}

The next theorem is the main result of this subsection.

\begin{thm}\label{thm-adm-gen}
   Let $(X,B+M)$ be a projective $\Q$-factorial g-sdlt pair with very general data. Let
   $m$ be a sufficiently large and divisible integer. Assume that $K_X +B+M$ is semi-ample. Then, $\mathcal{A}(X,m(K_X+B+M))$ generates $\mathcal{O}_X (m(K_X+B+M))$. 
\end{thm}

In order to prove Theorem \ref{thm-adm-gen}, we divide the argument into two steps. We first treat the next proposition.

\begin{prop}[\text{cf.\cite[Proposition 4.5]{fujino1}}]\label{prop-adm-gen}
	Let $(X,B+M)$ be a projective $\Q$-factorial g-dlt pair with very general data. Let
	$m$ be a sufficiently large and divisible integer. Assume that
	\begin{enumerate}
		\item $K_X +B+M$ is semi-ample.
		
		\item $\mathcal{A}(S,m(K_S+B_S+M_S))$ generates $\mathcal{O}_S(m(K_S+B_S+M_S))$, where $S=\rddown{B}$ and $K_S+B_S+M_S=(K_X+B+M)|_S$.
	\end{enumerate}
	Then the restriction $$\mathcal{PA}(X,m(K_X+B+M))\to \mathcal{A}(S,m(K_S+B_S+M_S))$$ is surjective, and $\mathcal{PA}(X,m(K_X+B+M))$ generates $\mathcal{O}_X (m(K_X+B+M))$.
\end{prop}

Let $f:X \to Y$ be the morphism induced by the semi-ample divisor $K_X +B+M$. We divided the proof of Proposition \ref{prop-adm-gen} into two cases.\\

\noindent \textbf{Case 1: $f|_S$ has connected fibres.}
\begin{lem}[\text{\cite[Lemma 4.4]{fujino1}}]\label{lem-contraction}
	Let $(X,B+M)$ be a g-klt pair, $S$ an effective $\Z$-divisor and $f: X\to Y$ a proper morphism with connected fibres such that $K_X+B+S+M\sim_\Q 0/Y$. If $S$ is vertical$/Y$, then $f|_S$ has connected fibres.
\end{lem}
\begin{proof}
	This is a special case of \cite[Theorem 3.1(2)]{birkar-connect}.
\end{proof}

\begin{lem}\label{lem-connected-fibre}
	Proposition \ref{prop-adm-gen} holds when $f|_S$ has connected fibres.
\end{lem}
\begin{proof}
	Let $g:S \to Z$ be the morphism given by the linear system $\mathcal{A}(S,m(K_S+B_S+M_S))$, and write $T=f(S)$. For any curve $C$ in the fibre of $f|_S$, $C$ is also contracted by $g$. Hence there is an induced morphism $T \to Z$ by the rigidity lemma.
	
	If $T=Y$, then for any section $s \in \mathcal{A}(S,m(K_S+B_S+M_S))$, there is a section $t$ on $Y$, and hence a section $f^*t \in \mathcal{PA}(X,m(K_X+B+M))$. Moreover, since $t$ generates $\mathcal{O}_Y(mH)$ where $H$ is an ample divisor of $Y$, we deduce $f^*t$ generates $\mathcal{O}_X (m(K_X+B+M))$.
	
	If $T \subsetneq Y$, then we considerthe following commutative diagram,
	 $$
	\xymatrix{
		H^0(X, \mathcal{O}_X (rm(K_X+B+M))) \ar[r]^{}   &   H^0(S, \mathcal{O}_S (rm(K_S+B_S+M_S)))   &\\
		H^0(Y, \mathcal{O}_Y(rmH)) \ar[u]^{\mathrm{id}}  \ar[r]^{} &     H^0(T, \mathcal{O}_T(rmH|_T))  \ar[u]^{\mathrm{id}} }
	$$
	Note that the horizontal arrows are surjective. By taking a sufficiently large number $r$, $\mathcal{O}_Y(rmH) \otimes \mathscr{I}_T$ is generated by global sections, where $ \mathscr{I}_T$ is the defining ideal of $T$. Hence the lemma follows.
\end{proof}

\noindent \textbf{Case 2: $f|_S$ has a disconnected fibre.}

\begin{rem}[An involution of a double cover]
	Let $f:X \to Y$ be a generically finite morphism of degree $2$. Then, $X$ has a natural nontrivial birational involution over $Y$ as follows. Consider the base change, where $(X \times_Y X)^\nu_{\text{main}}$ denotes the main components of the normalisation of $X \times_Y X$.
	$$
	\xymatrix{
		(X \times_Y X)^\nu_{\text{main}} \ar[d]_{} \ar[r]^{}   &  X\ar[d]^{f}\  &\\
		X \ar[r]^{f} &    Y } 
	$$
	Note that $(X \times_Y X)^\nu_{\text{main}} =X_1 \coprod X_2$ has two irreducible and reduced components, with $X_1 \cong X$ by the diagonal map, and $X_2$ is birational to $X$. Theorefore, the composite map $\iota: X \dashrightarrow X_2 \to X$ is the required involution, which permutes the two points of the general fibre.
\end{rem}

\begin{lem}[\text{cf.\cite[Lemma 2.3]{fujino1}}]\label{lem-horizontal-B}
	Let $(X,B+M)$ be a projective $\Q$-factorial g-lc pair with general data $\overline{M}$, and $g:X \to T$ be a $(K_X+B-\epsilon S +M)$-extremal contraction of relative dimension $1$, with $K_X+B+M \equiv 0/T$, where $S =\rddown{B}$ and $\epsilon>0$ is sufficiently small. Suppose the horizontal part $S^h$ intersects with a general fibre at two points. Then, there exists a $\Q$-rational function $\psi$, such that $\overline{M}+(\psi)$ is general, and either of the followings holds:
	\begin{enumerate}
		\item $S^h=S_1$ is irreducible, and the induced g-lc pair $(S_1^{\nu},  B_{S_1^{\nu}} +(M_{S_1^{\nu}}+(\psi|_{S_1})))$ has a $B$-birational involution over $T$.
		
		\item $S^h=S_1+S_2$ has two irreducible components, the induced g-lc pairs $(S_1^\nu,B_{S_1^\nu}+(M_{S_1^\nu}+(\psi|_{S_1^\nu})))$ and $(S_2^\nu,B_{S_2^\nu}+(M_{S_2^\nu}+(\psi|_{S_2^\nu})))$ are $B$-birational over $T$. 
	\end{enumerate}
\end{lem}
\begin{proof}
	By construction we have $\overline{M} \equiv 0$ over the generic point of $T$. Since $\overline{M}$ is general and $m\overline{M}$ is Cartier, we may assume $\overline{M}' = \overline{M}+(\psi)$ is vertical$/T$. In particular $\overline{M}$ is nef and abundant$/T$. By Lemma \ref{lem-nef-abundant-divisor}(2), there exists a birational model $T' \to T$ such that $\overline{M}' = 0/T'$, hence $\overline{M}'= g'^*M_{T'}$, where $g': \overline{X}' \to T'$. Therefore, the datum $\overline{M}|_{S_i'}+(\psi|_{S_i'})$ is the pull-back of $M_{T'}$ for $i=1,2$, where $S_i$ is the birational transform of $S_i$. Note that $K_X+B=0/T$ and $M+(\psi)=0/T$. So the lemma is concluded. 
\end{proof}

\begin{rem}\label{rem-horizontal}
	From the proof above, by \cite[Lemma 3.22]{hu2}, there exists a g-lc structure $(T,B_T+M_T)$ with data $\mathbf{M}_T$ such that $K_X+B+M \sim_\Q g^*(K_T+B_T+M_T)$ with $\mathbf{M}\sim_\Q g^*\mathbf{M}_T$.
\end{rem}

\begin{prop}[\text{cf.\cite[Proposition 2.1]{fujino1}}]\label{prop-disconnected-fibre}
	Let $(X,B+M)$ be a projective $\Q$-factorial g-dlt pair with very general data, and $f : X \to Z$ be a contraction onto a normal variety $Y$ with $K_X+B+M\sim_\Q 0/Y$. Suppose that the restriction $f|_S:S \to Y$ has a disconnected fibre, where $S:=\rddown{B}$. Then, there exists a rational map $g: X \dashrightarrow T/Y$ with general fibre $\PP^1$. Moreover, the horizontal$/T$ part $S^h$ is also horizontal$/Y$, and there exists a $\Q$-rational function $\psi$, such that $\overline{M}+(\psi)$ is general, and one of the followings holds:
	\begin{enumerate}
		\item $S^h$ is irreducible with mapping degree $[S^h:T]=2$, and there is a $B$-birational involution on $(S^h,B_{S^h}+(M_{S^h}+(\psi|_{S^h})))$ over $T$;
		
		\item $S^h=S_1+S_2$ with $S_1,S_2$ irreducible and the induced map $(S_1,B_{S_1}+(M_{S_1}+(\psi|_{S_1}))) \bir (S_2,B_{S_2}+(M_{S_2}+(\psi|_{S_2})))$ is $B$-birational over $T$.
	\end{enumerate}
    Furthermore, for any log resolution $\pi:(\overline{X},\overline{B}+\overline{M}) \to X$ such that the induced map $\overline{g}:\overline{X} \to T$ is a morphism, writing $\overline{S}=\overline{B}^{=1}$, then
    \begin{enumerate}
    	\item The restriction $\overline{g}|_{\overline{S}^v}:\overline{S}^v \to W \subsetneq T$ has connected fibres, where $\overline{S}^v$ denotes the vertical part; 
    	
    	\item For any irreducible component $W_j$ of $W$, there exists a stratum $V_j \subset \overline{S}^h$ dominant over $W_j$. 
    \end{enumerate}
\end{prop}
\begin{proof}
	By Lemma \ref{lem-contraction}, the horizontal$/Y$ part $\widetilde{S}^h \neq 0$, hence $K_X+B+M -\epsilon S$ is not pseudo-effective$/Y$ for an arbitrary small rational number $\epsilon>0$. Replacing $M$ with a general effective member, we run a $(K_X+B+M -\epsilon S)$-MMP$/Y$ which terminates with a Mori fibre space $g': X' \to T$ by \cite{bchm}. 
	By \cite[Proposition 4.37]{kollar-mmp}, if $f^{-1}(y) \bigcap S$ is disconnected, then the number of connected components of $f^{-1}(y) \bigcap S$ is $2$, and $(X, B + M)$ is plt over a neighborhood of $y$. Therefore, either the horizontal$/Y$ part $\widetilde{S}^h$ is irreducible or $\widetilde{S}^h=S_1+S_2$ with $S_i \to Y$ contractions. Moreover, by the argument of \cite[Proposition 4.37]{kollar-mmp}, the general fibre of $g'$ is $\PP^1$, $S^h=\widetilde{S}^h$, and $S^h \dashrightarrow T$ is generically finite of degree $2$. 
	The rest follows from Lemma \ref{lem-horizontal-B} and Lemma \ref{lem-slc-contraction}.
\end{proof}

\begin{rem}
	Notation as above, we remark the followings.
	\begin{enumerate}
		\item Since $g'$ is an extremal contraction, we deduce $T$ is a $\Q$-factorial klt variety, and it has a natural g-lc structure as in Remark \ref{rem-horizontal}.
		
		\item If further $\kappa_{\sigma}(K_X+B+M) =0$ and $S=S_1+S_2$, then the strict transform $S_i' \cong T$ for $i=1,2$ (see \cite[Proof of Proposition 2.1]{fujino1}). 
	\end{enumerate}	
\end{rem}

Recall that, given a line bundle $\mathcal{L}$ with a fixed associated Cartier divisor $s_L$, a \emph{meromorphic section} of $\mathcal{L}$ is $\psi\cdot s_L$ for some rational function $\psi \in k(X)$.
\begin{lem}[Descent of sections]\label{lem-descent}
	Let $f:X \to Z,g:Y \to Z$ be generically finite morphism of proper normal varieties and $L$ be a line bundle on $Z$. Given a base change, where $(X \times_Z Y)^\nu_{\text{main}}$ denotes the main components of the normalisation of $X \times_Z Y$,
	$$
	\xymatrix{
		(X \times_Z Y)^\nu_{\text{main}} \ar[d]_{\widetilde{f}} \ar[r]^{~~\widetilde{g}}   &  X\ar[d]^{f}\  &\\
		Y \ar[r]^{g} &    Z } 
	$$
	and meromorphic sections $s$ of $\mathcal{O}_X(f^*L)$ and $t$ of $\mathcal{O}_Y(g^*L)$ with $\widetilde{g}^*s=\widetilde{f}^*t$. Then, for a positive integer $m$ such that $\deg f |m$ and $\deg g|m$, there exists a meromorphic section $\sigma$ of $\mathcal{O}_Z(mL)$ such that $f^*\sigma =s^{\otimes m}$ and $g^*\sigma =t^{\otimes m}$.
\end{lem}
\begin{proof}
	Since $(s)\sim 0/Z$ and $(t) \sim 0/Z$, by \cite[Proposition 1.4, Example 1.7.4]{fulton}, there exist meromorphic sections $\sigma_1,\sigma_2 $ of $\mathcal{O}_Z(mL)$ such that $f_*(\frac{m}{\deg f}(s)) =  (\sigma_1)$, $g_*(\frac{m}{\deg g}(t) )=  (\sigma_2)$. Let $\varphi \in k(X)$ and $\psi \in k(Y)$ so that $(\varphi)=m(s) -f^*(\sigma_1)$ and $(\psi)= m(t) - g^*(\sigma_2)$. By \cite[Example 1.7.4]{fulton} again, we deduce that $N_{X/Z}(\varphi), N_{Y/Z}(\psi)$ are nowhere zero regular functions, hence constant functions. Moreover, since $\widetilde{g}^*s=\widetilde{f}^*t$, we deduce $(\sigma_1)=(\sigma_2)$, which in turn implies that $\widetilde{g}^*(\varphi)=\widetilde{f}^*(\psi)$, hence $(\varphi)=(\psi )=0$. Multiplying $\sigma_i$ with a suitable constant function, we obtain the required section. 
\end{proof}

\begin{lem}\label{lem-disconnected-fibre}
	Proposition \ref{prop-adm-gen} holds when $f|_S$ has a disconnected fibre.
\end{lem}
\begin{proof}
    There exists a rational map $g: X \dashrightarrow T/Y$ satisfying the conditions listed in Proposition \ref{prop-disconnected-fibre}. For any section $s \in \mathcal{A}(S,m(K_S+B_S+M_S))$, by Lemma \ref{lem-descent}, its restriction $(\psi^{m}\cdot s)|_{S^h}$ descends to a meromorphic section $\sigma$ of $\mathcal{O}_T(L)$. Now let $t=\psi^{-m}\cdot g^*\sigma$. It is obvious that $t|_{S^h}=s|_{S^h}$. It remains to verify that $t|_{S^v}=s|_{S^v}$, which completes the proof.
	
	Let $\pi:(\overline{X},\overline{B}+\overline{M}) \to X$ be a log resolution so that the induced map $\overline{g}:\overline{X} \to T$ is a morphism, and $\overline{t}=\pi^*t$.  By construction $\overline{t}=\overline{g}^* \sigma$. Since $\pi|_{\overline{S}}:\overline{S} \to S$ has connected fibres, where $\overline{S}=\overline{B}^{=1}$, it is sufficient to prove $\overline{s}|_{\overline{S}^v}=\overline{t}|_{\overline{S}^v}$, where $\overline{s}=\pi|_{\overline{S}}^*s$ and $\overline{S}^v$ denotes the vertical$/T$ part. To this end, by Proposition \ref{prop-disconnected-fibre}, $\overline{g}|_{\overline{S}^v}:\overline{S}^v \to W$ has connected fibres. Let $W_j \subset W$ be an irreducible component, and $S_j$ be its inverse image. By Proposition \ref{prop-disconnected-fibre}, there exists a stratum $V_j \subset \overline{S}^h$ dominant over $W_j$. Consider the following diagram
	 $$
	\xymatrix{
		H^0(S_j, m(K_X+B+M)|_{S_j}) \ar[d]^{p}\ar[r]^{}   &   H^0(V_j,  (m(K_X+B+M)|_{V_j})   &\\
		\mathcal{M}   \ar[r]^{\sim} &    \mathcal{M}  \ar[u]^{q} }
	$$
	where we regard $H^0(S_j, m(K_X+B+M)|_{S_j})$ as a sub-space of $k(S_j)$, and for any rational function $s \in H^0(S_i, m(K_X+B+M)|_{S_i})$, we define $p(s)=s \cdot \psi|_{S_j}^{-m} \in k(W_j)$. Moreover, we define $\mathcal{M}$ to be the image of $p$ in $k(W_j)$, and define $q(\sigma)=\overline{g}|_{V_j}^*\sigma \cdot \psi|_{V_j}^m$. It is easy to verify that the above diagram commutes and the right vertical arrow is injective. Therefore, the upper horizontal arrow in injective and the conclusion follows by $\overline{s}|_{V_i}=\overline{t}|_{V_i}$ for every $i$.
\end{proof}

\begin{proof}[Proof of Proposition \ref{prop-adm-gen}]
	Combing Lemma \ref{lem-connected-fibre} and Lemma \ref{lem-disconnected-fibre}.
\end{proof}

Now we treat the second step of the argument for Theorem \ref{thm-adm-gen}. 

\begin{lem}[\text{cf.\cite[Lemma 4.9]{fujino1}}]\label{lem-A-1}
	Let $(X,B+M)$ be a projective g-dlt pair with very general data and with $K_X +B+M$ semi-ample,
	and let $m$ be a sufficiently large and divisible integer. We write $G = \rho_m(\Bir(X,B,M))$.
	If $s\in\mathcal{PA}(X,m(K_X+B+M))$ and $S= \rddown{B}$, then $(\varphi^*s)|_S =s|_S$ and $\varphi^*s \in \mathcal{PA}(X,m(K_X+B+M))$ for every $\varphi\in G$.
	In particular, since $|G|$ is finite according to Theorem \ref{thmmain-good-lc}, we have
	$$
	\sum_{\varphi \in G} \varphi^*s \in \mathcal{A}(X,m(K_X+B+M)), 
	$$
	$$
    \prod_{\varphi \in G} \varphi^*s \in \mathcal{A}(X,m|G|(K_X+B+M)), 
	$$
	and
	$$
	(\prod_{\varphi \in G} \varphi^*s)|_S  = (s|_S)^{|G|}.
	$$
\end{lem}
\begin{proof}
	We can assume $X$ is irreducible. Replacing $(X,B+M)$ we can assume it is log smooth g-sub-dlt. Let $\alpha,\beta: (\overline{X},\overline{B}+\overline{M}) \to X$ be a common log resolution of $(\varphi:X \dashrightarrow X) \in G$. For every component $S_i \subset S=B^{=1}$, by Remark \ref{rem-B-bir} and making minor changes to  Lemma \ref{lem-induction}, there exists a stratum $V_i \subset S_i$, and a stratum $V_j$ such that $\varphi|_{V_i}$ is $B$-birational with  $$H^0(S_i, m(K_{S_i} + B_{S_i}+M_{S_i})) \simeq H^0(V_i, m(K_{V_i} + B_{V_i}+M_{V_i})),$$
	which in turn implies that $\alpha|_{\overline{S}}^* (s|_{S}) = \beta|_{\overline{S}}^* (s|_{S})$, where $\overline{S}=\overline{B}^{=1}$, hence $(\varphi^*s)|_S =s|_S$. The rest is trivial.
\end{proof}

\begin{lem}[\text{\cite[Lemma 4.7]{fujino1}}]\label{lem-A-2}
	Notation and assumptions as in Proposition \ref{prop-adm-gen}, $\mathcal{A}(X,l(K_X+B+M))$ generates $\mathcal{O}_X (l(K_X+B+M))$ for $l$ sufficiently large and divisible.
\end{lem}
\begin{proof}
	Let $\sigma_i$ be the $i$-th elementary symmetric polynomial and $n=|G|$. We obtain
	$$
	\{s=0 \} \supset \bigcap_{j=1}^{n}\{ \varphi_j^*s =0 \} = \bigcap_{i=1}^{n} \{ \sigma_i(\varphi_j^*s) \}.
	$$
	If $s \in \mathcal{PA}(X,m(K_X +B+M))$, then by Lemma \ref{lem-A-1}, we have
	$$
	\sigma_i^{n!/i}(\varphi_j^*s) \in \mathcal{A}(X,n!m(K_X+B+M)).
	$$
	Since $\mathcal{PA}(X,m(K_X +B+M))$ generates $\mathcal{O}_X (m(K_X+B+M))$ by Proposition \ref{prop-adm-gen}.
	Thus we can prove that $\mathcal{A}(X,n!m(K_X+B+M))$ generates $\mathcal{O}_X (n!m(K_X+B+M))$.
\end{proof}

\begin{proof}[Proof of Theorem \ref{thm-adm-gen}]
	By induction on the dimension of $X$, we can assume Theorem \ref{thm-adm-gen} holds for dimension $\le n-1$. By Lemma \ref{lem-descend}(1) we can replace $X$ with its normalisation $X^\nu$, hence assume $(X,B+M)$ is g-dlt. Finally we achieve the conclusion by Proposition \ref{prop-adm-gen} and Lemma \ref{lem-A-2}.  
\end{proof}

\subsection{Semi-ampleness for g-slc pairs}\label{sec4-2}

The main purpose of this subsection is to prove Theorem \ref{thmmain-semiample-slc}. We do not require the divisors to have $\Q$-coefficients.

\begin{lem}\label{lem-convex-com}
	Let $(X/Z,B+M)$ be a g-slc NQC pair with data $\overline{M}$ and $\nu : X^\nu \to X$ be the normalisation. Assume
	$K_{X^\nu} +B^\nu +M^\nu:= \nu^*(K_X +B+M)$ is semi-ample$/Z$. Then there exists convex combinations $B= \sum_i \alpha_i B^j$ and $\overline{M}= \sum_i \alpha_i \overline{M}^j$ such that
	\begin{enumerate}
		\item $(X,B^j+M^j)$ is a g-slc pair with data $\overline{M}^j$;
		
		\item $B^j$, $\overline{M}^j$ are $\Q$-divisors for every $j$;
		
		\item $K_{X^\nu} +B^{j,\nu} +M^{j,\nu}:= \nu^*(K_X +B^j+M^j)$ is semi-ample$/Z$ for every $j$.
	\end{enumerate}
   Moreover, if $\overline{M}$ is (log) abundant, then we can require $\overline{M}^j$ to be (log) abundant as well.
\end{lem}
\begin{proof}
	Let $\pi:(\overline{X},\overline{B}+\overline{M}) \to (X^\nu,B^\nu+M^\nu) $ be a log resolution, and write $\overline{X}=\coprod_i \overline{X}_i$ for the irreducible decomposition. For each $i$, there exists a rational polytope $\mathcal{Q}_i \subset \mathcal{P}_{B_i} \times \mathcal{P}_{M_i}$ containing the point $(\overline{B}_i,\overline{M}_i)$ such that $(X_i^\nu, \Delta_i^\nu +N_i^\nu)$ is g-dlt and $K_{X_i^\nu}+\Delta_i^\nu +N_i^\nu$ is semi-ample$/Z$ for every point $(\overline{\Delta}_i,\overline{N}_i) \in \mathcal{Q}$. Now consider the double locus $\overline{S} \le \overline{B}^{=1}$ and its irreducible decomposition $\overline{S}= \sum_k \overline{S}_{k,1} + \sum_k \overline{S}_{k,2} $, where 
	$$(\overline{S}_{k,1},B_{\overline{S}_{k,1}}+\overline{M}|_{\overline{S}_{k,1}}) \bir (\overline{S}_{k,2},B_{\overline{S}_{k,2}}+\overline{M}|_{\overline{S}_{k,2}})
	$$ is $B$-birational over $X$. Note that the restriction to $S_{k,j}$ gives an affine map of the polytope $\mathcal{Q}_i$. Taking intersections and modifying the polytopes we can assume $\mathcal{Q}_i|_{\overline{S}_{k,1}}=\mathcal{Q}_j|_{\overline{S}_{k,2}}$. Therefore, the vertices of these polytopes give the required convex combination.
	
	The last assertion follows from Lemma \ref{lem-convex-combination}.
\end{proof}

\begin{thm}[\text{cf.\cite[Theorem 1.4]{fujino-gongyo}\cite[Theorem 1.4]{haconxu}}]\label{thm-semiample-slc}
	Let $(X/Z,B+M)$ be a g-slc NQC pair and $\nu : X^\nu \to X$ be the normalisation. Assume
	$K_{X^\nu} +B^\nu +M^\nu:= \nu^*(K_X +B+M)$ is semi-ample$/Z$. Then $K_X +B+M$ is semi-ample$/Z$.
\end{thm}
\begin{proof}
	We can assume $Z=\Spec A$ is affine. By Lemma \ref{lem-convex-com} and Lemma \ref{lem-semi-ample-divisor}(1) for non-irreducible schemes, we can assume $B,\overline{M}$ are $\Q$-divisors. By Remark \ref{rem-g-data} we can assume $\overline{M}$ are general data. Combining Lemma \ref{lem-descend} and Theorem \ref{thm-adm-gen} we conclude the theorem.
\end{proof}

\begin{rem}
	Quite recently, a question is raised by K. Hashizume \cite{has-nonvanishing} that if the abundance conjecture in the setting of \cite[Theorem 1.3]{has-nonvanishing} holds. We believe it should be possible to achieved by Theorem \ref{thm-semiample-slc}.
\end{rem}

\subsection{Semi-ampleness for g-dlt pairs}\label{sec4-3}
Finally we close the paper with the following theorem. Also, we do not require the divisors to have $\Q$-coefficients.
\begin{thm}[\text{cf.\cite[Theorem 3.3]{hu}\cite[Corollary 1.5]{haconxu}}]\label{thm-lc-flips-3}
	Let $(X/Z, B+M)$ be a $\Q$-factorial g-dlt NQC pair with abundant data. Suppose that
	
	$\bullet$ $K_X + B+M$ is nef$/Z$,
	
	$\bullet$ $(K_X + B+M)|_{S_i}$ is semi-ample$/Z$ for each irreducible component $S_i$ of $\rddown{B}$,
	
	$\bullet$ $K_X + B - \epsilon P+M$ is semi-ample$/Z$ for some divisor $P \geq 0$ with $\mathrm{Supp} P = \rddown{B}$
	and for any sufficiently small number $\epsilon > 0$.
	
	Then, $K_X + B+M$ is semi-ample$/Z$.	
\end{thm}

An important ingredient of the proof of Theorem \ref{thm-lc-flips-3} is Diophantine approximation as follows.

\begin{lem}[\text{\cite[Lemma 3.7.7]{bchm}}]\label{diophantine}
	Let $\mathcal{C}$ be a rational polytope contained in a real vector
	space $V$ of dimension $n$, which is defined over the rationals. Fix a
	positive integer $k$ and a positive real number $\alpha$.
	If $v \in \mathcal{C}$, then we may find vectors $v_1$, $v_2$, $\dots$, $\in \mathcal{C}$ and positive
	integers $m_1$, $m_2$, $\dots$, $m_p$, which are divisible by $k$, such that $v$ is a convex
	linear combination of the vectors $v_1$, $v_2$, $\dots$, $v_p$ and
	\begin{align*}
		\|v_i-v\|\leq \frac{\alpha}{m_i} \mathrm{~where~} \frac{m_iv_i}{k} \mathrm{~is~integral}
	\end{align*}
	where $\|v_i-v\|$ denotes the norm of $v_i-v$ in $\R^n$.
\end{lem}

The following lemma is \cite[Lemma 3.6]{hu}. We present a proof for the reader's convenience.

\begin{lem}\label{semi-ample}
	Let $f:X \rightarrow Z$ be a projective morphism of normal varieties and $D$ be a divisor on $X$. Assume that $X$ is $\Q$-factorial, $Z$ is quasi-projective and $\mathbf{B}(|D/Z|_\R)= \emptyset$. Then, $D$ is semi-ample$/Z$. 
\end{lem}

\begin{proof}
	We can assume $Z=\Spec A$. By replacing $D$ with a member of $|D|_\R$, we can assume that $D$ is effective. Let $\sum_i D_i$ be the irreducible
	decomposition of $\mathrm{Supp} D $. For each $i$ there exists a divisor $M_i \sim_\R D$ such that $M_i$ does not contain $D_i$ in its support. More precisely, we write 
	$$
	D+ \sum_j a_{ij}(g_{ij})  =M_i
	$$
	where $a_{ij} \in \R$ and $g_{ij} \in k(X)$. Let $\sum_k M_{ik}$ be the irreducible
	decomposition of $\mathrm{Supp} M_i $, let $\mathcal{V}:= \sum_i \R_{\geq 0} D_i$, $\mathcal{W}_i= \sum_k \R_{\geq 0} M_{ik}$, and let $\mathcal{R}_i =\sum_j \R (g_{ij})$. It is easy to verify that there are rational polyhedrons 
	\begin{align*}
		\mathcal{L}_i \subset \{D' \in \mathcal{V}| D'+\mathcal{R}_i \mathrm{~intersects~with~}  \mathcal{W}_i \}
	\end{align*}
	containing $D$, and for every $\Q$-divisor $D' \in \mathcal{L} =\bigcap_i \mathcal{L}_i$, we have  $$\mathbf{B}(|D'|_\Q) \subset \bigcup_i (D_i \bigcap \mathrm{Supp}(M_i)). $$ 
	Note that $\dim D_i \bigcap \mathrm{Supp}(M_i) < \dim D_i$. Next we pick a divisor $M_i'$ for each $i$ such that $M_i'$ does not contain $B_i \bigcap \mathrm{Supp}(M_i)$ in its support. In a similar way we shrink $\mathcal{L}$ to a smaller polytope such that for every $\Q$-divisor $D' \in \mathcal{L} $, $$\mathbf{B}(|D'|_\Q) \subset \bigcup_i (D_i \bigcap \mathrm{Supp}(M_i) \bigcap  \mathrm{Supp}(M_i')) .$$ By repeating this process we obtain a rational polytope $\mathcal{L}$ containing $D$ such that every $\Q$-divisor $D' \in \mathcal{L}$ has $\mathbf{B}(|D'|_\Q)=\emptyset$, hence $D'$ is semi-ample. Therefore, $D$ is a convex combination of semi-ample $\Q$-divisors, which concludes the lemma by Lemma \ref{lem-semi-ample-divisor}.
\end{proof}


\begin{proof}[Proof of Theorem \ref{thm-lc-flips-3}]
	We can assume $Z= \Spec A$, $X$ is projective over $Z$, and divide the proof into two steps. If $B$ has $\Q$-coefficients and $M$ is $\Q$-Cartier, then one can skip Step $1$ and directly go to Step $2$. Otherwise, by arguments of Step $1$, we reduce to the situation that we can apply Diophantine approximation.
	
	{\noindent \textbf{Step 1.}} Let $S=\lfloor B\rfloor$ and $\sum_i B_i$ be the irreducible
	decomposition of $\mathrm{Supp} \Delta$ where $\Delta=\{B\}$ is the fractional part of $B$. Now fix $\overline{M} \ge 0$, such that $(X,\Delta+M)$ is klt, and let $\overline{M}=\sum_j \alpha_j \overline{M}_j$ be the convex combination of nef $\Q$-Cartier divisors. Replacing $\overline{M}_j$, we can assume $\overline{M}_j \ge 0$. By Lemma \ref{lem-convex-combination}, we can assume that, for every $i$,
	\begin{itemize}
		\item $\overline{M}_j$ is nef and abundant, and
		
		\item $\overline{M}_j$ is supported by $\Supp \overline{M}$.
	\end{itemize} 
    We put $\mathcal{V}:=(\sum_i [0,1] B_i) \times \mathcal{M}$ where $\mathcal{M}$ is the rational polytope given by the convex combination of $\{ \overline{M}_j\}$. It is easy to check that
	$$
	\mathcal{L} =\{(\Theta,\overline{N}) \in \mathcal{V} |(X, S+\Theta+N)\text{ is g-lc with data $\overline{N}$, and } K_X+S+\Theta+N \mathrm{~is~nef}/Z  \}
	$$
	is a rational polytope containing $(\Delta,\overline{M})$ (cf.\cite[Proposition 3.16]{hanli}\cite[Proposition 3.2]{birkar-existII}).

	Let $S_i$ be a component of $S$. There exists an $\R$-rational function $\varphi_i(\overline{M})$, depending on $\overline{M}$, such that $\overline{M}+(\varphi_i(\overline{M}))$ does not contain the strict transform of $S_i$ in its support. Moreover, by Lemma \ref{lem-convex-combination}, and possibly shrinking $\mathcal{M}$, there is an affine surjective map $\mathcal{M} \to \mathcal{M}_i$ of rational polytopes, by sending $\overline{N}$ to $\overline{N}+(\varphi_i(\overline{N}))$, such that any $\overline{N}' \in \mathcal{M}_i$ does not contain $S_i$ in its support.
	
	Since $K_{S_i}+B_{S_i}+M_{S_i}= (K_X+B+M+(\varphi_i(\overline{M})))|_{S_i}$ is semi-ample$/Z$, there is a morphism $g_i:S_i \rightarrow V_i$ where $V_i$ is the ample model. So, there is an ample$/Z$ divisor $A_i$ and a finite number of rational functions $h_{ij}$ on $S_i$ such that
	$$
	K_{S_i}+B_{S_i}+M_{S_i}+\sum_j a_{ij}(h_{ij})= g_i^\ast A_i 
	$$ where $a_{ij}\in \R$. 
	Let $T_i=\lfloor B_{S_i} \rfloor$ and $\sum_k B_{S_i,k}$ be the irreducible
	decomposition of $\mathrm{Supp}\{ B_{S_i}\}$. Define $\mathcal{W}_i=(\sum_k [0,1]B_{S_i,k}) \times \mathcal{M}_{S_i}$, where $\mathcal{M}_{S_i}=\mathcal{M}_i|_{S_i}$, and define $\mathcal{R}_i=\sum_j \R(h_{ij})$ and $\mathcal{A}_i \subset \mathrm{Div}_\R(V_i)$ as a sufficiently small rational polytope containing $A_i$. One easily verifies that 
	\begin{align*}
		\mathcal{P}_i := \{ (\Theta_{S_i}, \overline{N}_{S_i}) \in \mathcal{W}_i|& (S_i, T_i +\Theta_{S_i}+N_{S_i})\text{ is g-lc with data $\overline{N}_{S_i}$, and } \\
		& K_{S_i}+T_i +\Theta_{S_i}+N_{S_i}+ \mathcal{R}_i\mathrm{~intersects~with~}g_i^\ast \mathcal{A}_i \}
	\end{align*}
	is a rational polytope in $\mathcal{W}_i$. It follows that
	\begin{align*}
		\mathcal{L}_i := \{& (\Theta,\overline{N}) \in \mathcal{L}| (\Theta_{S_i}, (\overline{N}+(\varphi_i(\overline{N})))|_{S_i}) \in \mathcal{P}_i  \}
	\end{align*}
	is also a rational sub-polytope containing $(\Delta,\overline{M})$. Replacing $\mathcal{L}$ with $\bigcap_i \mathcal{L}_i$, where $i$ runs over all components $S_i$ of $G$, for each point $(\Theta,\overline{N}) \in \mathcal{L}$, we have that $(X,S+\Theta+N)$ is g-lc with data $\overline{N}$, and that $(K_X+S+\Theta+N)|_{S_i}$ is semi-ample$/Z$ for every $i$.\\
	
	{\noindent \textbf{Step 2.}} Let $\beta=\|\Delta + M\|$. Let $d$ be an integer such that $dK_X$, $d B_i$ are Cartier for all $i$, $d M_{j,l}$ are Cartier for all $l$ where $M$ is supported by $\sum_l M_{j,l}$, and pick $\alpha \ll 1- \beta$. Write the convex combinations $\Delta=\sum_i \alpha_i \Delta_i$ and $\overline{M}=\sum_i \alpha_i \overline{M}_i$. By Lemma \ref{diophantine} we can further assume that 
	\begin{align*}
		\|\Delta_i-\Delta \|\leq \frac{\alpha}{m_i},~ \|M_i-M \|\leq \frac{\alpha}{m_i}\mathrm{~where~} \frac{m_i}{d}\Delta_i,\frac{m_i}{d}M_i \mathrm{~are~integral}.
	\end{align*}
   Write
	\begin{align*}
		m_k(K_X+S+\Delta_k+M_k)-S = &K_X+\Delta +M +m_k(\Delta_k-\Delta)+m_k(M_k-M) \\
		& +(m_k-1)\epsilon P+(m_k-1)(K_X+B+M-\epsilon P). 
	\end{align*}
	Since $(X,\Delta +M +m_k(\Delta_k-\Delta)+m_k(M_k-M) +(m_k-1)\epsilon P)$ is klt for $0<\epsilon \ll 1$, $S$ is supported by $(m_k-1)\epsilon P$ and $K_X+B-\epsilon P$ is semi-ample$/Z$, by Kawamata-Viehweg vanishing and an injectivity theorem \cite[Theorem 6.1]{fujino-inj}, we have that
	$$
	H^1(X, \mathcal{O}_X(m_k(K_X+S+\Delta_k+M_k)-S)) \rightarrow H^1(X,\mathcal{O}_X(m_k(K_X+S+\Delta_k+M_k)))
	$$
	is injective, hence 
	$$
	H^0(X, \mathcal{O}_X(m_k(K_X+S+\Delta_k+M_k))) \rightarrow H^0(S,\mathcal{O}_S(m_k(K_S+B_{k,S}+M_{k,S})))
	$$ 
	is surjective. We have the commutative diagram as follows.
	$$
	\xymatrix{
		H^0(X, \mathcal{O}_X(m_k(K_X+S+\Delta_k+M_k))) \ar[d]_{} \ar[r]_{}&  H^0(S,\mathcal{O}_S(m_k(K_S+B_{k,S}+M_{k,S}))) \ar[d]^{}  &\\
		\mathcal{O}_X(m_k(K_X+S+\Delta_k+M_k))  \ar[r]_{}  & \mathcal{O}_S(m_k(K_S+B_{k,S}+M_{k,S}))   &
	} 
	$$
	Since the upper arrow and the right arrow are surjective, by Theorem \ref{thm-semiample-slc}, $\mathcal{O}_X(m_k(K_X+S+\Delta_k+M_k)) $ is relatively
	globally generated along $S$ over $Z$ which in turn implies that the $\R$-stable base locus $\mathbf{B}(|K_X+B+M/Z|_\R)$ does not intersect with $S$. Since $K_X+B+M-\epsilon P$ is semi-ample$/Z$ for $0 <\epsilon \ll 1$, the $\R$-stable base locus $\mathbf{B}(|K_X+B+M/Z|_\R)$ must be contained in the support of $S$, hence $\mathbf{B}(|K_X+B+M/Z|_\R) =\emptyset$. We apply Lemma \ref{semi-ample} to conclude that $K_X+B+M$ is semi-ample$/Z$.
\end{proof}

\begin{rem}
	For Theorem \ref{thm-lc-flips-3}, we can loosen the semi-ampleness for each component of $\rddown{B}$ and the assumption $\Supp P =\rddown{B}$ to the semi-ampleness for each component of $\Supp P$ and  $\Supp P \subseteq \rddown{B}$. But note that in this case we may need to require an extra assumption on the base locus of $\overline{M}$. See \cite[Theorem 3.3]{hu} for a reference and apply the latter assertion of Lemma \ref{lem-convex-combination}. 
\end{rem}



\begin{thebibliography}{BCHM}

\bibitem[ADK13]{adk} D.~Abramovich, J.~Denef, K.~Karu. Weak toroidalization over non-closed
fields. Manuscripta Math., {\textbf{142}} (2013) (1-2), 257--271.

	

\bibitem[AK00]{ak} D.~Abramovich, K.~Karu, Weak semistable reduction in characteristic $0$, Invent. math. {\textbf{139}} (2000), no. 2, 241--273.




\bibitem[Amb04]{ambro1} F.~Ambro, Shokurov's boundary property. J. Differential Geom. {\textbf{67}} (2004), no. 2, 229-255.











\bibitem[Bir11]{birkar-existII}C.~Birkar, On existence of log minimal models I\!I, J. Reine Angew Math. {\textbf{658}} (2011), 99--113.


\bibitem[Bir12]{birkar-flip} C.~Birkar, Existence of log canonical flips and a special LMMP, Publ. Math. Inst. Hautes \'Etudes Sci. {\textbf{115}} (2012), no. 1, 325--368.

\bibitem[Bir20]{birkar-connect}C.~Birkar, On connectedness of non-klt loci of singularities of pairs, preprint (2020), arXiv:2010.08226.



\bibitem[BCHM10]{bchm}C.~Birkar, P.~Cascini, C.~D.~Hacon, J.~M\textsuperscript{c}Kernan, Existence of minimal models for varieties of log general type, J. Amer. Math. Soc. {\textbf{23}}(2010), no. 2, 405--468.


\bibitem[BH14]{bhzariski} C.~Birkar, Z.~Hu, Polarized pairs, log minimal models, and Zariski decompositions, Nagoya Math. J. {\textbf{215}} (2014), 203--224. 






\bibitem[BZh16]{birkarzhang} C.~Birkar, D.~Q.~Zhang, Effectivity of Iitaka fibrations and pluricanonical systems of polarized pairs, Publ. Math. Inst. Hautes \'Etudes Sci. {\textbf{123}} (2016), no. 1, 283--331.












\bibitem[Fuj00]{fujino1} O.~Fujino, 
Abundance theorem for semi log canonical threefolds, 
Duke Math. J. {\textbf{102}} (2000), no. 3, 513--532.









\bibitem[Fuj14]{fujino-slc} O.~Fujino, Fundamental theorems for semi log canonical pairs, Algebr. Geom. \textbf{1} (2014), no. 2, 194-228.

\bibitem[Fuj17]{fujino-inj} O.~Fujino, Injectivity theorems. Higher Dimensional Algebraic Geometry: In honour of Professor Yujiro Kawamata's sixtieth birthday, Adv. Stud. Pure Math., \textbf{74}, Math. Soc. Japan (2017), 131--157.

\bibitem[Fuj-book]{fujino-book}O.~Fujino, {\em Foundations of the minimal model program}, MSJ Mem. \textbf{35}, Mathematical Society in Japan, Tokyo, (2017). 








\bibitem[FG14]{fujino-gongyo}O.~Fujino, Y.~Gongyo, Log pluricanonical representations and abundance conjecture, Compos. Math. {\textbf{150}} (2014) no. 4, 593--620.










\bibitem[Ful83]{fulton} W.~Fulton, {\em Intersection Theory}. Springer, Berlin (1983). 











\bibitem[HX13]{haconxu-lcc}C.~D.~Hacon, C.~Xu, Existence of log canonical closures, Invent.Math. {\textbf{192}} (2013), no. 1, 161--195. 

\bibitem[HX16]{haconxu}C.~D.~Hacon, C.~Xu, On finiteness of $B$-representation and semi-log canonical abundance, Adv, Stud. Pure. Math. {\textbf{70}} (2016), Minimal Models and extremal rays--Kyoto, 2011, 361--378. 


\bibitem[HL18]{hanli}J.~Han, Z.~Li, Weak Zariski decompositions and log terminal models for generalized polarized pairs. Preprint (2018).


\bibitem[Harts]{hartshorne} R.~Hartshorne, {\em{Algebraic geometry}}, Graduate Texts in Mathematics, No. {\textbf{52}}. Springer-Verlag, New York-Heidelberg, 1977.






\bibitem[Ha20]{has-nonvanishing} K.~Hashizume, Non-vanishing theorem for generalized lc pairs with a polarization, preprint (2020), arXiv:2012.15038.


\bibitem[HH20]{hashizumehu}K.~Hashizume, Z.~Hu, On minimal model theory for log abundant lc pairs, J. Reine Angew.
Math., \textbf{767} (2020), 109--159.


\bibitem[Hu17]{hu}Z.~Hu, Log canonical pairs with boundaries containing ample divisors, preprint (2017), arXiv:1712.07219.  

\bibitem[Hu20a]{hu2}Z.~Hu, Log abundance of the moduli b-divisors of lc-trivial fibrations, preprint (2020), arXiv:2003.14379.  

\bibitem[Hu20b]{hu3}Z.~Hu, Existence of canonical models for Kawamata log terminal pairs, preprint (2020), arXiv:2004.03895.
  
















\bibitem[Ko13]{kollar-mmp}J.~Koll\'ar, {\em Singularities of the Minimal Model Program}, Cambridge Tracts in Mathematics, {\textbf{200}}. Cambridge University Press, Cambridge, 2013.




\bibitem[KM98]{kollar-mori} J.~Koll\'ar, S.~Mori, {\em{Birational geometry of algebraic varieties}}, With the collaboration of C.~H.~Clemens and A.~Corti. Translated from the 1998 Japanese original. Cambridge Tracts in Mathematics, {\textbf{134}}. Cambridge University Press, Cambridge, (1998).











\bibitem[Nak04]{nakayama}N.~Nakayama, {\em Zariski-decomposition and abundance}, MSJ Mem., {\textbf{14}}, Mathematical Society of Japan, Tokyo, (2004). 











\bibitem[U75]{ueno} K.~Ueno, {\em Classification theory of algebraic varieties and compact complex
spaces}, Lecture Notes in Math., Vol. \textbf{439}, Springer, Berlin, (1975).


\end{thebibliography}
\end{document}